\documentclass{article}

\usepackage{amsmath}       
\usepackage{amsthm}        
\usepackage{amssymb}       
\usepackage{braket}       
\usepackage{amsfonts}
\usepackage[all]{xy}
\usepackage{xspace}
\usepackage{calc}
\usepackage{comment}
\usepackage{pgfplots}
\usepgflibrary{fpu}
\usepackage{mathtools}
\usepackage{tabularx}
\usepackage{ifthen}             
\usepackage{graphicx}
\usepackage{docmute}
\usepackage{array}
\usepackage{subfloat} 
\usepackage{textcomp}
\usepackage{bbm}            
\usepackage{etoolbox}  
\usepackage{verbatim}
\usepackage{stmaryrd}
\usepackage{cancel}
\usepackage{xcolor}
\colorlet{Black}{black}
\usepackage{tensor}
\usepackage{enumerate}  
\usepackage{thm-restate}
\usepackage[numbers,sort]{natbib}
\usepackage{enumitem}
\usepackage{tocloft}
\usepackage{pdfpages}
\usepackage{soul}
\usepackage[export]{adjustbox}

\usepackage[margin=3cm]{geometry}

\makeatletter
\def\calign@preamble{%
   &\hfil\strut@
    \setboxz@h{\@lign$\m@th\displaystyle{##}$}%
    \ifmeasuring@\savefieldlength@\fi
    \set@field
    \hfil
    \tabskip\alignsep@
}
\let\cmeasure@\measure@
\patchcmd\cmeasure@{\divide\@tempcntb\tw@}{}{}{}
\patchcmd\cmeasure@{\divide\@tempcntb\tw@}{}{}{}
\patchcmd\cmeasure@{\ifodd\maxfields@
  \global\advance\maxfields@\@ne
  \fi}{}{}{}    
\newenvironment{calign}
{%
  \let\align@preamble\calign@preamble
  \let\measure@\cmeasure@
  \align
}
{%
  \endalign
}  
\makeatother
\tikzset{
    master/.style={
        execute at end picture={
            \coordinate (lower right) at (current bounding box.south east);
            \coordinate (upper left) at (current bounding box.north west);
        }
    },
    slave/.style={
        execute at end picture={
            \pgfresetboundingbox
            \path (upper left) rectangle (lower right);
        }
    }
}
\newcommand{\Tr}{\mathrm{Tr}}

\theoremstyle{plain} %%% Plain Theorem Styles.

\newtheorem{theorem}{Theorem}[section]
\newtheorem{lemma}[theorem]{Lemma}
\newtheorem{corollary}[theorem]{Corollary}          
\newtheorem{proposition}[theorem]{Proposition}

\newtheorem*{theorem*}{Theorem}
\newtheorem*{proposition*}{Proposition}
\newtheorem*{definition*}{Definition}
\newtheorem*{corollary*}{Corollary}

\theoremstyle{definition} %%%% Definition-like Commands  
\newtheorem{definition}[theorem]{Definition}
\newtheorem{remark}[theorem]{Remark}

\newtheorem{notation}[theorem]{Notation}

\newtheorem{problem}[theorem]{Problem}

\theoremstyle{remark}  %%%% Remark-like Commands
\newtheorem{example}[theorem]{Example}

\newtheoremstyle{special_statement} 
        {\topskip}% Space above
        {\topskip}% Space below
        {\addtolength{\leftskip}{2.5em} \itshape }% Body font
        {}% Indent amount % \parindent
        {\bfseries}% Theorem head font
        {:}% Punctuation after theorem head
        {.5em}% Space after theorem head
        {}% Theorem head spec (can be left empty, meaning `normal')
\theoremstyle{special_statement}

\DeclareMathOperator{\Hom}{Hom}
\DeclareMathOperator{\End}{End}

\newcommand{\id}{\mathrm{id}}

\newcommand{\Bimod}{\mathrm{Bimod}}

\newcommand{\Aut}{\ensuremath{\mathrm{Aut}}}

\newcommand{\Rep}{\mathrm{Rep}}
\newcommand{\Mat}{\mathrm{Mat}}

\newcommand{\QRel}{\mathrm{QRel}}
\newcommand{\Rel}{\mathrm{Rel}}
\newcommand{\Stoch}{\mathrm{Stoch}}
\newcommand{\Chan}{\mathrm{Chan}}
\newcommand{\Mod}{\mathrm{Mod}}

\newcommand{\Ann}{\mathrm{Ann}}

\newcommand{\CP}{\ensuremath{\mathrm{CP}}}

\newcommand{\Hilb}{\ensuremath{\mathrm{Hilb}}}
\newcommand{\TwoHilb}{\ensuremath{\mathrm{2Hilb}}}
\newcommand{\TwoFHilb}{\ensuremath{\mathrm{2FHilb}}}
\newcommand{\TwoRep}{\ensuremath{\mathrm{2Rep}}}

\newcommand{\F}{\ensuremath{\mathrm{SSFA}}}

\usepackage[colorlinks=true, linkcolor=black, citecolor=black, urlcolor=black, hypertexnames=false
        ]{hyperref}

\newcommand\ignore[1]{}
\usepackage[utf8]{inputenc}
\usepackage{authblk}

\title{Covariant quantum combinatorics with applications to zero-error communication}
\author{Dominic Verdon\thanks{dominic.verdon@bristol.ac.uk}}
\affil{School of Mathematics, University of Bristol}
\date{\today}

\begin{document}
\normalsize
\maketitle 

\begin{abstract}
We develop the theory of quantum (a.k.a. noncommutative) relations and quantum (a.k.a. noncommutative) graphs in the finite-dimensional covariant setting, where all systems (finite-dimensional $C^*$-algebras) carry an action of a compact quantum group $G$, and all channels (completely positive maps preserving a certain $G$-invariant functional) are covariant with respect to the $G$-actions. We motivate our definitions by applications to zero-error quantum communication theory with a symmetry constraint. Some key results are the following: 1) We give a necessary and sufficient condition for a covariant quantum relation to be the underlying relation of a covariant channel. 2) We show that every quantum confusability graph with a $G$-action (which we call a quantum $G$-graph) arises as the confusability graph of a covariant channel. 3) We show that a covariant channel is reversible precisely when its confusability $G$-graph is discrete. 4) When $G$ is quasitriangular (this includes all compact groups), we show that covariant zero-error source-channel coding schemes are classified by covariant homomorphisms between confusability $G$-graphs.
\end{abstract}

\section{Introduction}

\paragraph{Group symmetry in quantum information.}

Group symmetry constraints arise in quantum information theory for a variety of reasons. It is well-known that reference frame uncertainty~\cite{Bartlett2007} and superselection rules in the physical theory~\cite{Weinberg1995} induce group symmetry constraints. However, there are many other places in quantum information theory where group symmetry constraints appear naturally. It is common to consider constraints on states arising from a group action; for instance, the symmetric subspace arising from the action of $S_n$ on the Hilbert space of $n$ qudits~\cite{Harrow2013}. More generally, it is also common to consider constraints on channels arising from a group action on the source and target system (see~\cite{Datta2017} and references therein). Channels constrained in this way are called \emph{covariant}. More recently, compact quantum group symmetry has appeared in quantum information theory, for example in the classification of entanglement-assisted strategies for nonlocal games~\cite{Musto2019} and the construction of new families of quantum channels~\cite{Brannan2020}.

Because group symmetry constraints are so frequently considered in quantum information, it is useful to develop a theory which is itself \emph{covariant} (a.k.a. \emph{equivariant}), in the sense that the constructions within it are compatible with covariance constraints arising from the action of some symmetry group. It turns out that the project of categorical quantum mechanics initiated in~\cite{Abramsky2004}, in particular as developed in~\cite{Heunen2019,Verdon2021}, is an extremely useful tool for formulating such a theory. This is because the important structure of the category of finite-dimensional Hilbert spaces and linear maps is that it is a \emph{rigid $C^*$-tensor category}. If the constructions relevant to the theory can be formulated purely in terms of the categorical structure, then, since the categories of finite-dimensional (f.d.) continuous unitary representations of compact quantum groups (which include all ordinary compact groups) are also rigid $C^*$-tensor categories, covariance is immediate.

In this work we demonstrate this approach to covariance by formulating the  theory of quantum relations and quantum graphs using purely categorical methods. This theory, which is the combinatorial/possibilistic theory underlying the analytic/probabilistic theory of quantum mechanics, is relevant to zero-error quantum communication~\cite{Duan2012,Stahlke2015}, just as the theory of ordinary relations and graphs (which is subsumed by the noncommutative/quantum theory) is relevant to zero-error classical communication~\cite{Shannon1956}.

\paragraph{Quantum relations and quantum graphs.} We will now define the main mathematical objects of our study, in language hopefully accessible to the working quantum information theorist. In non-covariant finite-dimensional (f.d.) quantum theory, systems are f.d.\ $C^*$-algebras equipped with a choice of faithful positive linear functional, and channels are completely positive functional-preserving linear maps.

\begin{itemize}
\item A \emph{finite quantum set} is a finite set of f.d.\ Hilbert spaces~\cite[Def. 2.1]{Kornell2020}. It is well-known that every f.d.\ $C^*$-algebra is of the form $\bigoplus_{i =1}^{n} B(H_i)$ for some $n \in \mathbb{N}$; that is, every f.d.\ $C^*$-algebra is a direct sum of matrix algebras. Every f.d.\ $C^*$-algebra may therefore be associated with a finite quantum set, namely the set of Hilbert spaces corresponding to its factors.

\item Let $A,B$ be two f.d.\ $C^*$-algebras associated to quantum sets $S_A := \{H_1,\dots,H_m\}$ and $S_B:= \{K_1,\dots,K_n\}$ respectively. A \emph{quantum relation} $A \to B$ is a choice of subspace $L_{ij} \subseteq B(H_i,K_j)$ for all pairs $H_{i},K_j$. 

\item A \emph{quantum graph} on an f.d.\ $C^*$-algebra $A$ is a quantum relation $R: A \to A$ which is \emph{symmetric}: this is to say that $L_{ij} = L_{ji}^{\dagger}$. 
\end{itemize}
It is easy to check that quantum relations and quantum graphs reduce to ordinary relations and graphs when the $C^*$-algebras are commutative (i.e. when all factors are 1-dimensional).

\paragraph{Covariant quantum relations and quantum $G$-graphs.} In the covariant setting, some group $G$ is fixed. The most general kind of groups we will consider in this paper are compact quantum groups, but for accessibility we will restrict ourselves to ordinary compact groups $G$ for now. 

Systems are f.d.\ $C^*$-algebras equipped with some \emph{action} of the compact group $G$ (we call these \emph{$G$-$C^*$-algebras}, but they are also known as \emph{$C^*$-dynamical systems}). An action of $G$ on an f.d.\ $C^*$-algebra $A$ is a continuous homomorphism $\alpha: G \to \Aut(A)$. We will write $\alpha_g := \alpha(g)$. Every system $A$ is equipped with a faithful positive linear functional $\phi: A \to \mathbb{C}$ which is  \emph{$G$-invariant}; that is, it satisfies $\phi(x) = \phi(\alpha_g(x))$ for all $x \in A, g \in G$. In this work we will make a canonical choice of such functional for every system, which we call the \emph{separable standard} functional (Appendix~\ref{app:tworepdef}); however, as we explain in Remark~\ref{rem:functionalchoice}, the choice of $G$-invariant functional does not make a substantial difference to the overall theory. All completely positive maps $f: A \to B$ must be \emph{covariant}; that is to say that $f(\alpha_{A,g}(x)) = \alpha_{B,g}( f(x))$ for all $x \in A, g \in G$. \emph{Channels} are completely positive maps preserving the chosen functional, in the sense that $\phi_B(f(x)) = \phi_A(x)$.
\ignore{
Recall that an automorphism of an f.d.\ $C^*$-algebra $\bigoplus_{i =1}^{n} B(H_i)$ is specified by a pair $(\sigma,\{U_1,\dots,U_n\})$, where $\sigma \in S_n$ is a permutation which exchanges factors of equal dimension, and $U_i \in B(H_i)$ is a unitary operator acting on the factor $H_i$. With this notation, we can straightforwardly define covariance for quantum relations:
\begin{itemize}
\item Let $A,B$ be $G$-$C^*$-algebras. For any $g \in G$, let $$\alpha_{A,g} = (\sigma_{A,g},\{U_{A,1,g},\dots,U_{A,m,g}\}),~~~~~~~~~~~\alpha_{B,g} = (\sigma_{B,g},\{U_{B,1,g},\dots,U_{B,n,g}\})$$ be the corresponding automorphisms of $A$ and $B$. A \emph{covariant quantum relation} $A \to B$ is a quantum relation $A \to B$ satisfying:
$$
L_{\sigma_{A,g}(i),\sigma_{B,g}(j)} = U_{B,j,g} \circ L_{ij}  \circ U_{A,i,g}^{\dagger}~~~~~~\forall ~(i,j) \in\{1,\dots,m\} \times \{1,\dots,n\},~ g \in G
$$ 
\item A \emph{quantum $G$-graph} on an f.d.\ $G$-$C^*$-algebra $A$ is a covariant, symmetric quantum relation $A \to A$.
\end{itemize}
}

Using the definitions already given it is straightforward to define covariance for quantum relations; since we do not want to overburden the reader with definitions at this point we will omit the details. We will call a covariant quantum graph a \emph{quantum graph with a $G$-action}, or more succinctly a \emph{quantum $G$-graph}. These definitions reduce to the classical notions of covariant relations and graphs with a $G$-action when the $C^*$-algebras are commutative.

\paragraph{Zero-error communication.} Having defined the mathematical objects which are the subject of this paper, we can now highlight their physical significance. Recall that a quantum channel $f: A=  \bigoplus_{i =1}^m B(H_i) \to \bigoplus_{j =1}^n B(K_j)= B$ is defined by a set of \emph{Kraus maps} $\{f_{ijk}: H_i \to K_j\}_{k \in R_{ij}}$ for each pair of factors; here $\{R_{ij}\}$ are index sets for the Kraus maps.

\begin{definition*}
Let $f: A=  \bigoplus_{i =1}^m B(H_i) \to \bigoplus_{j =1}^n B(K_j)= B$ be a covariant channel. Let $\mathfrak{R}(f)$ be the covariant quantum relation obtained by setting $L_{ij} := \mathrm{span}\{f_{ijk}\}_{k \in R_{ij}}$. We say that $\mathfrak{R}(f)$ is the \emph{quantum relation underlying the channel}. 
\end{definition*}
\noindent
One can consider a quantum channel to be a stochastic mixture of its Kraus maps; then the underlying quantum relation encodes the possibilistic structure of the channel, i.e. what inputs can possibly get mapped to which outputs. This information is what matters for zero-error communication.

We will now define a second quantum combinatorial structure from a covariant channel. Observe that for any covariant quantum relation $L: A \to B$ there is a \emph{converse} covariant quantum relation $L^{\dagger}: B \to A$ defined by 
$$(L^{\dagger})_{ji} := \{f^{\dagger}~|~f \in L_{ij}\} \subseteq B(K_j,H_i).$$
Moreover, we can compose two covariant quantum relations $L: A \to B$ and $R: B \to C$ by
$$(R \circ L)_{ik} := \textrm{span}\{g \circ f~|~ j \in J, g \in R_{jk},f \in L_{ij}\}$$

\begin{definition*}
Let $f: A \to B$ be a covariant channel and let $\Gamma_f$ be the quantum $G$-graph on $A$ defined by  $\Gamma_f:= \mathfrak{R}(f)^{\dagger} \circ \mathfrak{R}(f).$ We say that $\Gamma_f$ is the \emph{confusability graph of the channel $f$}.
\end{definition*}
\noindent
When the $C^*$-algebras $A,B$ are matrix algebras, we recover the noncommutative confusability graphs of~\cite[\S{}2]{Duan2012}; when these algebras are commutative, we recover the confusability graphs of~\cite[Fig. 3]{Shannon1956}.

\subsection{Our results}

We now state our results. In what follows, let $G$ be any compact quantum group. 

\paragraph{Quantum relations.}
We saw above that one can take the converse of a covariant quantum relation, and compose covariant quantum relations. With this structure, covariant quantum relations between $G$-$C^*$-algebras form a dagger category $\QRel(G)$. Let $\CP(G)$ be the dagger category whose objects are $G$-$C^*$-algebras and whose morphisms are covariant completely positive (CP) maps.
\begin{proposition*}[Prop~\ref{prop:relfct}]
Taking the underlying quantum relation of a covariant CP map defines a full unitary (i.e. dagger-preserving) functor $\mathfrak{R}: \CP(G) \to \QRel(G)$.
\end{proposition*}
\noindent
Although every covariant quantum relation is the underlying relation of a covariant CP map, it is not necessarily the underlying relation of a covariant channel, since not all CP maps preserve the separable standard functional. By the covariant Choi's theorem~\cite[Thm. 4.13]{Verdon2021}, covariant quantum relations $A \to B$ correspond bijectively to projections in a certain $C^*$-algebra $\underline{\mathrm{Hom}}(A,B)$.
\begin{proposition*}[Prop.~\ref{prop:relchancond}]
A covariant quantum relation is the underlying relation of a covariant channel iff the partial trace of its associated projection $\pi \in \underline{\mathrm{Hom}}(A,B)$ is invertible.
\end{proposition*}

\paragraph{Quantum graphs.}
We say that a $G$-graph is \emph{simple} if its associated projection is orthogonal to a certain fixed projection. We say that the complement of a simple $G$-graph is a \emph{confusability $G$-graph}. The confusability graph of any covariant channel (defined above) is a confusability $G$-graph; the following proposition then justifies the name, and answers a question raised by Daws~\cite[P.26]{Daws2022}.
\begin{proposition*}[Prop.~\ref{prop:graphtochan}]
Every confusability $G$-graph is the confusability graph of a covariant channel. 
\end{proposition*}
\noindent
We remark that even when the confusability $G$-graph is classical (i.e. the $G$-$C^*$-algebra is commutative), the covariant channel constructed in the proof of this proposition may have a noncommutative $G$-$C^*$-algebra as its target. We do not know whether the proposition would still hold if one were only to consider classical channels. 

\paragraph{Reversibility of channels.} It is very useful to know whether a covariant channel $f: A \to B$ is reversible; that is, whether there exists a covariant channel $g: B \to A$ such that $g \circ f = \id_{A}$. The following necessary and sufficient condition provides an agreeably simple and general (but finite-dimensional) complement to previous answers in the non-covariant setting~\cite{Nayak2007,Jencova2012,Shirokov2013}.
\begin{theorem*}[Thm.~\ref{thm:reversal}]
A covariant channel $f: A \to B$ is reversible iff its confusability $G$-graph is discrete. 
\end{theorem*}

\paragraph{Covariant zero-error source-channel coding.} Quantum zero-error source-channel coding problems were defined in~\cite[\S{}4]{Stahlke2015}; a long list of problems which can be formulated in terms of zero-error source-channel coding was given in that work. 

We will now define a generalisation of zero-error source-channel coding to the covariant setting. This is nothing more than \emph{coherent QSCC} as defined in~\cite[Fig. 3]{Stahlke2015}, but where the systems under consideration are generalised from matrix algebras to $G$-$C^*$-algebras, and all channels involved must be covariant.  In order to define a tensor product of arbitrary systems we restrict to the case where the compact quantum group is quasitriangular.
\begin{problem}
Alice and Bob share a covariant communication channel $N: A \to B$, where $A$ and $B$ are systems. Charlie wants to send the state of a system $S$ to Bob. To do this, he will transmit information to Alice (a state of a system $O_A$) and some `side information' to Bob (a state of a system $O_B$). This transmission is defined by a covariant channel $C: S \to O_A \otimes O_B$. Alice must use the covariant channel $N$ to transmit information to Bob in order that Bob can recover the original state of the system $S$. 

Such a procedure is defined by an covariant encoding channel $E: O_A \to A$  and a covariant decoding channel $D: B \otimes O_B \to S$. If the data $(E,D)$ is a solution to the problem which will transmit any state of $S$ perfectly we call it a \emph{covariant zero-error source channel coding scheme}.
\end{problem} 
\noindent
A schematic outline is given in the following diagram:
\begin{align*}
\includegraphics[valign=c]{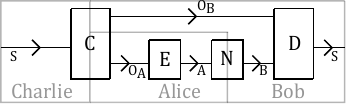}
\end{align*}
We here show that covariant zero-error source-channel coding schemes correspond to covariant homomorphisms of quantum $G$-graphs. For confusability $G$-graphs $\Gamma_A,\Gamma_B$ on systems $A,B$, we say that a covariant channel $f: A \to B$ is a \emph{homomorphism} if $\mathfrak{R}(f)^{\dagger} \circ \Gamma_B \circ \mathfrak{R}(f) \subseteq \Gamma_A$, where the inclusion indicates containment of the subspaces associated to pairs of factors. When $A,B$ are commutative this reduces to a `stochastic homomorphism' of ordinary graphs, as described in Example~\ref{ex:stochgraphhom}. (As we discuss in the subsequent remark, by restricting from channels to $*$-cohomomorphisms one can define a stronger notion of homomorphism which reduces to the usual notion for ordinary graphs, but this is too strong for applications to zero-error source-channel coding.)

Every covariant source $C: S \to O_A \otimes O_B$ is associated to a confusability $G$-graph on $O_A$, called the \emph{confusability $G$-graph of the source}.
\begin{theorem*}[Thm.~\ref{thm:scchoms}]
A covariant channel $E: O_A \to A$ is a valid encoding channel for a covariant zero-error source-channel coding scheme with covariant source $C: S \to O_A \otimes O_B$ and covariant communication channel $N: A \to B$ precisely when it is a covariant homomorphism from the confusability $G$-graph of the source to the confusability $G$-graph of the communication channel.
\end{theorem*}
\noindent
Conceptually, one might wonder whether this gives an operational semantics for covariant homomorphisms between confusability $G$-graphs. The following proposition and corollary answer this question in the affirmative. 
\begin{proposition*}[Prop.~\ref{prop:channelfromsourcegraph}]
Every confusability $G$-graph is the confusability $G$-graph of some covariant source. 
\end{proposition*}
\begin{corollary*}
Let $G_1,G_2$ be confusability $G$-graphs. The set of covariant homomorphisms $G_1 \to G_2$ is precisely the set of encoding channels for a covariant zero-error source-channel coding scheme where the confusability $G$-graph of the source is $G_1$ and the confusability $G$-graph of the channel is $G_2$.
\end{corollary*}

\subsection{Related work}

\paragraph{A covariant Lovasz theta number?} In the paper~\cite{Duan2012}, a key result is the existence of a quantum Lovasz theta number; this is a real-valued function on quantum confusability graphs which is computable by a semidefinite programme, monotone under entanglement-assisted homomorphisms~\cite[Thm. 19]{Stahlke2015}, and multiplicative with respect to the tensor product. If the definition could be phrased in terms of categorical structure then it should generalise to a `covariant Lovasz number' for quasitriangular compact quantum groups $G$; that is, a real-valued function on quantum confusability $G$-graphs satisfying the same properties.

\paragraph{Higher quantum theory.} This paper can be seen as a continuation of the programme of higher quantum theory initiated in~\cite{Vicary2012,Vicary2012a}. Here we show that more general quantum protocols such as zero-error source-channel coding can also be formulated using monoidal 2-categories, and also give a further motivation (covariance) for this 2-categorical formulation. 

\paragraph{Operator algebras in rigid $C^*$-tensor categories.} Our results hold for systems and channels in an arbitrary rigid $C^*$-tensor category. To understand the following explanation the reader may find it helpful to read Section~\ref{sec:background} and Appendix~\ref{app:tworepdef} first. Let $\mathcal{T}$ be a rigid $C^*$-tensor category. If $\mathcal{T} \simeq \Rep(G)$ for some compact quantum group $G$, then $\F$s in $\mathcal{T}$ may be identified with $G$-$C^*$-algebras equipped with their separable standard functional, as described in Appendix~\ref{app:tworepdef}. In the general case (for there are rigid $C^*$-tensor categories which are not equivalent to representation categories of compact quantum groups), $\F$s may be identified with finite-dimensional \emph{operator algebras in the rigid $C^*$-tensor category $\mathcal{T}$} in the sense of~\cite{Jones2017,Jones2017a}. The Q-system completion of $\mathcal{T}$ in the sense of Appendix~\ref{app:tworepdef} yields a semisimple $C^*$-2-category which we will call $\TwoRep(\mathcal{T})$, in which $\mathcal{T}$ embeds as the endomorphism category of a simple object $r_0$.

\ignore{Semisimple $C^*$-2-categories in the sense of~\cite[Def. 2.25]{Verdon2021} are additive $C^*$-2-categories which are locally semisimple, rigid (i.e. all 1-morphisms are dualisable), where every object is a finite direct sum of objects with simple identity $1$-morphism, and where all $\F$s in endomorphism categories split as pair of pants algebras. By~\cite[Prop. 3.25]{Verdon2021}, semisimple $C^*$-2-categories are precisely the Q-system completions of rigid $C^*$-tensor categories.

We will be precise about how our results generalise;} 

By semisimplicity of $\TwoRep(\mathcal{T})$ every $\F$ in $\mathcal{T}$ is an algebra of the form $X \otimes X^*$ for some 1-morphism out of the object $r_0$ in $\TwoRep(\mathcal{T})$, with multiplication and unit as in~\eqref{eq:pairofpants}. By~\cite[Thm. 4.11]{Verdon2021} CP morphisms and channels between $\F$s in $\mathcal{T}$ may be defined by their dilations as in Proposition~\ref{prop:dilation}, and Choi's theorem (Theorem~\ref{thm:choi}) follows. Relations and graphs for systems in $\mathcal{T}$ may be defined precisely as in Section~\ref{sec:relsandgraphs}, since we only used categorical structure in the definition. The results up to Section~\ref{sec:scc} then hold for systems, channels, graphs and relations in $\mathcal{T}$, using precisely the same proofs. If the category $\mathcal{T}$ is additionally braided, the results in Section~\ref{sec:scc} hold also, again using precisely the same proofs. The material in Appendix~\ref{app:monoidal} also goes through for a general braided rigid $C^*$-tensor category, using precisely the same proofs.

\subsection{Data availability statement}

Data sharing not applicable to this article as no datasets were generated or analysed during the current study.

\subsection{Acknowledgements}

We are especially grateful to Matthew Daws and Andre Kornell for independently pointing out an error in Proposition~\ref{prop:relchancond} of the published version of this article; the erroneous part of that proposition is struck through in this version, and a correction will be sent to the journal. We thank Benjamin Musto, David Reutter and Andreas Winter for useful conversations, and Ashley Montanaro for his support of this work. All diagrams were drawn using the open source vector graphics program \emph{Inkscape}. This work has been funded by the European Research Council (ERC) under the European Union’s Horizon 2020 research and innovation programme (grant agreement No. 817581). This work has also been funded by EPSRC. 

\section{Background}\label{sec:background}

\subsection{The 2-category $\TwoRep(G)$}

In this section we give a very quick summary of the 2-category $\TwoRep(G)$ for a compact quantum group $G$, and its graphical calculus. (For a more detailed introduction see~\cite{Verdon2021}.) We assume familiarity with the definition of a 2-category~\cite{Johnson2021}, as well as the definition of dagger categories and unitary (a.k.a. dagger) functors~\cite{Heunen2019a}. We use the notation $\circ$ for vertical composition and $\otimes$ for horizontal composition. By coherence for bicategories we need not worry about whether the 2-categories we consider are strict~\cite[\S{}4]{Bartlett2008}, so we use `2-category' to refer to both strict and weak bicategories.

\paragraph{Graphical calculus for 2-categories.}
The graphical calculus for 2-categories is a region-labelled version of the standard `string diagram' calculus for monoidal categories. Objects are represented by labelled regions; 1-morphisms are represented by labelled wires; and 2-morphisms are represented by boxes with labelled input and output wires. Wires corresponding to identity 1-morphisms, and boxes corresponding to identity 2-morphisms, are invisible. Diagrams which are planar isotopic represent the same 2-morphism. 

We read diagram composition from left to right and from bottom to top. Vertical composition of 2-morphisms is represented by vertical juxtaposition of diagrams. For instance, let $X,Y,Z: r \to s$ be 1-morphisms and $\alpha: X \to Y$, $\beta: Y \to Z$ be 2-morphisms; then $\beta \circ \alpha: X \to Z$ is represented as follows:
\begin{calign}\nonumber
\includegraphics[valign=c]{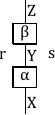}
\end{calign}
Horizontal composition of 2-morphisms is represented by horizontal juxtaposition of diagrams. For instance, let $X,X': r \to s$ and $Y,Y': s \to t$ be 1-morphisms, and let $\alpha: X \to X'$, $\beta: Y \to Y'$ be 2-morphisms; then $\alpha \otimes \beta: X \otimes Y \to X' \otimes Y'$ is represented as follows:
\begin{calign}\nonumber
\includegraphics[valign=c]{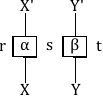}
\end{calign}

\paragraph{The 2-category $\TwoRep(G)$.} The semisimple $C^*$-2-category $\TwoRep(G)$ can be defined in at least three equivalent ways~\cite{Verdon2021}.
\begin{itemize}
\item The 2-category of separable standard Frobenius algebras, dagger bimodules and bimodule homomorphisms in the category $\Rep(G)$ of continuous f.d.\ unitary representations of $G$ (also known as the \emph{Q-system completion} of $\Rep(G)$~\cite{Giorgetti2023,Chen2022}). This 2-category was called $\Bimod(\Rep(G))$ in~\cite{Verdon2021}.
\item The 2-category of cofinite semisimple finitely decomposable $\Rep(G)$-$C^*$-module categories, unitary linear module functors and natural transformations of module functors. This 2-category was called $\Mod(\Rep(G))$ in~\cite{Verdon2021}.
\item The 2-category of f.d.\ $G$-$C^*$-algebras~\cite{Wang1998}, $G$-equivariant f.d.\ Hilbert $C^*$-bimodules~\cite[P.6]{Neshveyev2018} and $G$-equivariant bimodule homomorphisms. 
\end{itemize}
Since the reader may be unfamiliar with these definitions we have given a self-contained summary of the first definition in Appendix~\ref{app:tworepdef}. The equivalence between the first and second definition was shown in~\cite[Thm. 3.21]{Verdon2021}; a sketch of how the third definition is equivalent to the first is given in Appendix~\ref{app:tworepdef}.

Throughout this work we will write objects with lower case letters $r,s,\dots$; 1-morphisms with upper case letters $X,Y,\dots$; and 2-morphisms with lower case letters $f,g,\dots$.

\paragraph{The $C^*$ structure of $\TwoRep(G)$.} The 2-category $\TwoRep(G)$ is a \emph{$C^*$-2-category}. In particular:
\begin{itemize}
\item For any 1-morphisms $X,Y: r \to s$ the set $\Hom(X,Y)$ is a Banach space. Vertical and horizontal composition induce linear maps on 2-morphism spaces.
\item Every 2-morphism $f:X \to Y$ has a \emph{dagger} 2-morphism $f^{\dagger}: Y \to X$. Taking the dagger induces an antilinear map on Hom-spaces. The dagger satisfies the following properties:
\begin{align}
(f^{\dagger})^{\dagger} = f 
&&
(f \otimes g)^{\dagger} = f^{\dagger} \otimes g^{\dagger}
&&
||f^{\dagger} \circ f|| = ||f||^2
\end{align}
The last property implies that, for any 1-morphism $X$, the algebra $\End(X)$ is a $C^*$-algebra with involution given by the dagger. 
\item For any 2-morphism $f: X \to Y$, the 2-morphism $f^{\dagger} \circ f$ is a positive element of the $C^*$-algebra $\End(X)$.
\end{itemize}
We say that a 2-morphism $f: X \to Y$ is an \emph{isometry} if $f^{\dagger} \circ f = \id_X$, a \emph{coisometry} if $f \circ f^{\dagger} = \id_Y$, and a \emph{unitary} if it is both an isometry and a coisometry. We say that a 2-morphism $f: X \to Y$ is a \emph{partial isometry} if $f^{\dagger} \circ f \in \End(X)$ is a projection (or equivalently, if $f \circ f^{\dagger} \in \End(Y)$ is a projection).

\paragraph{Rigidity of $\TwoRep(G)$.} The 2-category $\TwoRep(G)$ is \emph{rigid}. This means that every 1-morphism $X: r \to s$ has a \emph{dual} 1-morphism $X^*: s \to r$. In order to represent duality we orient the 1-morphism wires: $X$ and $X^*$ are represented by a wire with upwards and downwards pointing arrows respectively. Duality is characterised by the following 2-morphisms, called \emph{cups} and \emph{caps}:
\begin{calign}\nonumber
\includegraphics[valign=c]{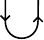}
&
\includegraphics[valign=c]{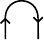}
&
\includegraphics[valign=c]{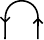}
&
\includegraphics[valign=c]{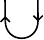} 
\\\label{eq:cupscaps}
\eta_X: \id_s \to X^* \otimes X
&
\epsilon_X: X \otimes X^* \to \id_r
&
\eta_X^{\dagger}: X^* \otimes X \to \id_s
&
\epsilon_X^{\dagger}: \id_r \to X \otimes X^*
\end{calign}
These cups and caps obey the \emph{snake equations}:
\begin{calign}\label{eq:snake}
\includegraphics[valign=c]{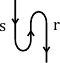}
~~=~~
\includegraphics[valign=c]{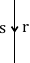}
&
\includegraphics[valign=c]{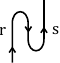}
~~=~~
\includegraphics[valign=c]{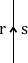}
&
\includegraphics[valign=c]{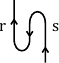}
~~=~~
\includegraphics[valign=c]{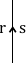}
&
\includegraphics[valign=c]{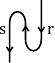}
~~=~~
\includegraphics[valign=c]{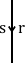}
\end{calign}

\paragraph{Linear structure of $\TwoRep(G)$.} The 2-category $\TwoRep(G)$ is  \emph{locally semisimple}, which means that, for any pair of objects $r,s$:
\begin{itemize}
\item For any pair of 1-morphisms $X,Y: r \to s$ there is a 1-morphism $X_1 \oplus X_2: r \to s$ (called the \emph{direct sum}), with isometries $i_1: X_1 \to X_1 \oplus X_2$, $i_2: X_2 \to X_1 \oplus X_2$ such that $i_1 \circ i_1^{\dagger} + i_2 \circ i_2^{\dagger} = \id_{X_1 \oplus X_2}$. 
\item There is a \emph{zero 1-morphism} ${\bf 0}_{r,s}: r \to s$ such that $\End({\bf 0})$ is zero-dimensional.
\item For any 1-morphism $X: r \to s$, we say that a 2-morphism $f \in \End(X)$ is a \emph{projection} (a.k.a. a \emph{dagger idempotent}) if $f = f^{\dagger} = f \circ f$. Every projection $f \in \End(X)$ has a \emph{splitting}, i.e. a 1-morphism $V: r \to s$ together with an isometry $\iota_f: V \to X$ such that $f = \iota_f \circ \iota_f^{\dagger}$. Splittings are unique up to unitary isomorphism.
\item The $C^*$-algebra $\End(X)$ is finite-dimensional for every 1-morphism $X$. (Together with idempotent splitting this implies that every 1-morphism is a finite direct sum of \emph{simple} 1-morphisms, i.e. 1-morphisms $X_i$ such that $\End(X_i) \cong \mathbb{C}$.)
\end{itemize}
Every pair of objects also has a direct sum, in the following sense. 
\begin{definition}
We say that a \emph{direct sum} of two objects $r_1, r_2$ is an object $r_1 \boxplus r_2$ with inclusion and projection 1-morphisms $\iota_i: r_i \to r_1 \boxplus r_2$, $\rho_i: r_1 \boxplus r_2 \to r_i$ such that:
\begin{itemize}
\item $\iota_i \otimes \rho_i$ is unitarily isomorphic to $\id_{r_i}$.
\item $\iota_1 \otimes \rho_2 \in \Hom(r_1,r_2)$ and $\iota_2 \otimes \rho_1 \in \Hom(r_2,r_1)$ are zero 1-morphisms.
\item $\id_{r_1 \oplus r_2}$ is a direct sum of $\rho_1 \otimes \iota_1$ and $\rho_2 \otimes \iota_2$.
\end{itemize}
\end{definition}

\paragraph{Dagger, transpose and conjugate in $\TwoRep(G)$.} For any 2-morphism $f: X \to Y$ we can define its \emph{transpose} (a.k.a. \emph{mate}) $f^*: Y^* \to X^*$: 
\begin{calign}\label{eq:rtranspose}
\includegraphics[valign=c]{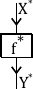}
~~=~~
\includegraphics[valign=c]{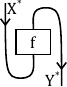}
\end{calign}
We can also define its \emph{conjugate} $f_{*}: X^* \to Y^*$:
$$
f_{*} := (f^*)^{\dagger} = (f^{\dagger})^*
$$
To represent all these 2-morphisms in the diagrammatic calculus we draw 2-morphism boxes with an offset edge:
\begin{calign}\label{eq:boxesdefine}
\includegraphics[valign=c]{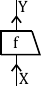}
&
\includegraphics[valign=c,scale=1]{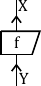}
&
\includegraphics[valign=c]{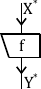}
&
\includegraphics[valign=c,scale=1]{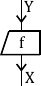}
\\
f: X \to Y
&
f^{\dagger}: Y \to X 
&
f^*: Y^* \to X^*
&
f_{*}: X^* \to Y^*
\end{calign}
The 2-morphisms then obey the following \emph{sliding equations}:
\begin{calign}
\includegraphics[valign=c]{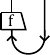}
~~=~~
\includegraphics[valign=c]{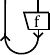}
&
\includegraphics[valign=c]{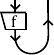}
~~=~~
\includegraphics[valign=c]{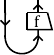}
&
\includegraphics[valign=c]{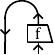}
~~=~~
\includegraphics[valign=c]{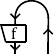}
&
\includegraphics[valign=c]{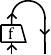}
~~=~~
\includegraphics[valign=c]{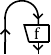}
\\\label{eq:sliding}
\includegraphics[valign=c]{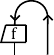}
~~=~~
\includegraphics[valign=c]{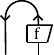}
&
\includegraphics[valign=c]{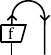}
~~=~~
\includegraphics[valign=c]{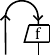}
&
\includegraphics[valign=c]{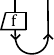}
~~=~~
\includegraphics[valign=c]{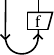}
&
\includegraphics[valign=c]{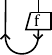}
~~=~~
\includegraphics[valign=c]{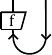}
\end{calign}

\paragraph{Dimension and trace in $\TwoRep(G)$.}
For any 1-morphism $X: r \to s$ we define its \emph{left dimension} to be the 2-morphism $\dim_L(X):= \eta_X^{\dagger} \circ \eta_X \in \End(\id_s)$, and its \emph{right dimension} to be the 2-morphism $\dim_R(X):= \epsilon_X \circ \epsilon_X^{\dagger} \in \End(\id_r)$.
\begin{calign}
\dim_L(X) ~=~ 
\includegraphics[valign=c]{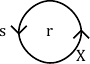}
~~~~~~~~~~~~~~~~
\dim_R(X) ~=~ 
\includegraphics[valign=c]{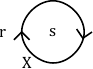}
\end{calign}
Note that the $C^*$-algebras $\End(\id_s)$ and $\End(\id_r)$ are both commutative, and the elements $\dim_L(X)$ and $\dim_R(X)$ are positive. There is a canonical choice of faithful positive trace $\Tr_s: \End(\id_s) \to \mathbb{C}$ and $\Tr_r: \End(\id_r)\to \mathbb{C}$ on these commutative  algebras (given in Definition~\ref{def:endotrace}); we define the \emph{quantum dimension} of the 1-morphism $X$ to be the number $d(X):= \Tr_s(\dim_L(X)) = \Tr_r(\dim_R(X))$.

If $\dim_L(X)$ and $\dim_R(X)$ are invertible we say that $X$ is a \emph{separable} 1-morphism. We write $d_X := \dim_L(X)$ and $n_{X}:= \sqrt{d_X} \in \End(\id_s)$.

For any 1-morphism $X: r \to s$ we define the \emph{trace} $\Tr: \End(X) \to \mathbb{C}$ as follows:
\begin{calign}\label{eq:qtrace}
\Tr(f):=
\Tr_s\left(
\includegraphics[valign=c]{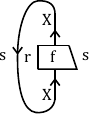}\right)
~=~
\Tr_r\left(\includegraphics[valign=c]{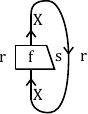}\right)
\end{calign}
As the notation suggests, this is a positive faithful trace on $\End(X)$ (this, and the equality between the left and right traces in~\eqref{eq:qtrace}, are shown in Lemma~\ref{lem:endotrace}). Note that $d(X) = \Tr(\id_X)$.

\paragraph{Tensor product in $\TwoRep(G)$.} If $G$ is a quasitriangular compact quantum group then $\TwoRep(G)$ inherits a tensor product (see Appendix~\ref{app:monoidal}). In this case we can use a 3-dimensional diagrammatic calculus~\cite[\S{}8]{Heunen2019}\cite{Hummon2012}.  Let $X,X': r_1 \to r_2$ and $Y,Y': s_1 \to s_2$. In the diagrammatic calculus, the tensor product $f \boxtimes g: X \boxtimes Y \to X' \boxtimes Y'$ of two 2-morphisms $f: X \to X'$ and $g: Y \to Y'$ is depicted by layering $f$ below $g$:
\begin{align*}
\includegraphics[valign=c]{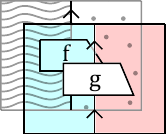}
\end{align*}
In  the above diagram the $g$-plane is in front of the $f$-plane. The $s_1$-region is shaded in blue, the $s_2$ region is shaded in red, the $r_1$ region is shaded with wavy lines, and the $r_2$ region is shaded with polka dots. Throughout this work we will follow the convention of shading the plane in front with translucent colours, and the plane at the back with black and white patterns. 

The unit object for this tensor product is not depicted in the graphical calculus (i.e. the corresponding planar regions are left white). 

For an object $r \boxtimes s$ and a morphism $X \in \End(r)$ (resp. $Y \in \End(s)$) there exists a positive faithful partial trace $\Tr_s(f): \End(X \boxtimes \id_s) \to \End(X)$  such that $\Tr_r \circ \Tr_s = \Tr_{r \boxtimes s}$. (See Definition~\ref{def:partialtrace} and Lemma~\ref{lem:partialtrace}.)

\begin{example}[The case of trivial $G$.] If the group $G$ is trivial then the 2-category $\TwoRep(G)$ is precisely the 2-category $\TwoFHilb$ defined in~\cite[\S{}8.2.5]{Heunen2019}. To see this, one can use the second definition of $\TwoRep(G)$ in terms of cofinite semisimple finitely decomposable $C^*$-module categories; when the group $G$ is trivial, the module and cofiniteness conditions disappear and we are left only with the 2-category of semisimple finitely decomposable $C^*$-categories; these are all finite direct sums of copies of the $C^*$-category $\Hilb$, whereby we arrive at the definition of $\TwoFHilb$. 

The 2-category $\TwoFHilb$ admits an interpretation in terms of indexed families of Hilbert spaces and linear maps which was developed in~\cite[\S{}8.2]{Heunen2019}~\cite{Reutter2019}\cite{Verdon2022}. The equivalence classes of objects of $\TwoFHilb$ correspond to natural numbers; equivalence classes of 1-morphisms $X,Y,\dots: [m] \to{} [n]$ correspond to $m \times n$ matrices of f.d. Hilbert spaces; and the 2-morphisms $f,g,\dots:X \to Y$ correspond to $m \times n$ matrices of linear maps between these f.d. Hilbert spaces. Composition of 1-morphisms is matrix multiplication using the tensor product and the direct sum of Hilbert spaces. Horizontal composition of 2-morphisms is matrix multiplication using the tensor product and direct sum of linear maps, and vertical composition of 2-morphisms is entrywise composition of linear maps. 
\end{example}

\begin{example}[The case of a finite group $G$.] If $G$ is a finite group then we can use the second definition of $\TwoRep(G)$ and the classification of $\Rep(G)$-module categories given in~\cite{Ostrik2003}. Equivalence classes of simple objects of $\TwoRep(G)$ --- i.e. Morita equivalence classes of simple f.d. $G$-$C^*$-algebras --- correspond to conjugacy classes of pairs $(H,\psi)$, where $H < G$ is a subgroup and $\psi \in H^2(H_1,\mathbb{C}^*)$ is a cohomology class. A description of the $\Hom$-categories between two such simple objects is given in~\cite[Prop. 3.1]{Ostrik2003}. The $\Hom$-categories between direct sums of simple objects are given by matrices of 1-morphisms between the simple factors and matrices of 2-morphisms between these 1-morphisms. In fact, this is generally true; once one knows the $\Hom$-categories between equivalence classes of simple objects, the $\Hom$-categories between direct sums of simple objects can be defined using matrix algebra~\cite[App. 6]{Verdon2021}.
\end{example}

\begin{example}[The case of a torsion-free compact quantum group] We say that a compact quantum group $G$ is \emph{torsion-free} if the category $\Rep(G)$ of f.d. unitary representations of $G$ is torsion free in the sense of~\cite[Def. 3.7]{Arano2015}. For torsion-free $G$, the 2-category $\TwoRep(G)$ has one equivalence class of simple objects, so, just as for trivial $G$, equivalence classes of objects correspond to natural numbers counting the number of times the unique simple object appears as a factor in the direct sum. Equivalence classes of 1-morphisms $X,Y,\dots: [m] \to{} [n]$ correspond to $m \times n$ matrices of f.d. unitary representations of $G$. The 2-morphisms $f,g,\dots: X \to Y$ are matrices of intertwiners between these representations. Composition of 1-morphisms is matrix multiplication using the direct sum and tensor product of representations.
\end{example}

\subsection{Systems and channels}\label{sec:systemschannels}

Let $G$ be a compact quantum group. As we discuss in Appendix~\ref{app:tworepdef}, a finite-dimensional (f.d.) $G$-$C^*$-algebra $A$ equipped with its canonical separable standard linear functional can be identified with a \emph{separable standard Frobenius algebra (\F)} in the category $\Rep(G)$ of f.d. unitary representations of $G$. This is an object $A$ of $\Rep(G)$ equipped with a \emph{multiplication} morphism $m: A \otimes A \to A$ and a \emph{unit} morphism $u: \mathbbm{1} \to A$ (where $\mathbbm{1}$ is the trivial representation), satisfying the following conditions:   
\begin{align*}
\includegraphics[scale=.8,valign=c]{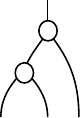}
~~=~~
\includegraphics[scale=.8,valign=c]{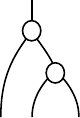}
&&
\includegraphics[scale=.8,valign=c]{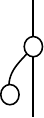}
~~=~~
\includegraphics[scale=.8,valign=c]{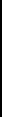}
~~=~~
\includegraphics[scale=.8,valign=c]{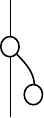}
&&
\includegraphics[scale=.8,valign=c]{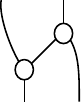}
~~=~~
\includegraphics[scale=.8,valign=c]{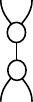}
~~=~~
\includegraphics[scale=.8,valign=c]{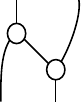}
\end{align*}
Here we have drawn $m: A \otimes A \to A$, $u: \mathbbm{1} \to A$ and their adjoints the \emph{comultiplication} $m^{\dagger}: A \to A \otimes A$ and \emph{counit} $u^{\dagger}: A \to \mathbbm{1}$ as white vertices; they can be distinguished by their type. There are also additional conditions, \emph{separability} and \emph{standardness}, which are detailed in Appendix~\ref{app:tworepdef} but will not be necessary for the discussion here. We therefore make the following definition.
\begin{definition}
We define a \emph{system} to be an $\F$ in $\Rep(G)$.
\end{definition}
\noindent
By~\cite[Thm. 3.2.3]{Verdon2020b} (c.f.~\cite[Thm. 7.18]{Heunen2019}), a covariant CP map between systems $A,B$ can be identified with a morphism $f: A \to B$ in $\Rep(G)$ such that the following is a positive element of the f.d. $C^*$-algebra $\End(A \otimes B)$:
\begin{calign}\label{eq:cpcond}
\includegraphics[scale=.8,valign=c]{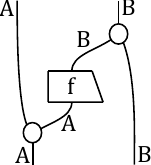}
\end{calign}
Here the white vertex at the bottom left is the comultiplication of $A$, while the white vertex at the top right is the multiplication of $B$. 
\begin{definition}
Let $A,B$ be systems. We call a morphism $f: A \to B$ in $\Rep(G)$ obeying~\eqref{eq:cpcond} a \emph{CP morphism}. If the morphism $f$ obeys the additional counit-preservation condition
\begin{calign}\label{eq:tp}
u_B^{\dagger} \circ f = u_A^{\dagger}
\end{calign}
we say that it is a \emph{channel}. This additional condition~\eqref{eq:tp} says precisely that $f$ preserves the canonical separable standard functional on the f.d. $G$-$C^*$-algebras. 

Systems and CP morphisms (resp. channels) form a category which we call $\CP(G)$ (resp. $\Chan(G)$). The category $\CP(G)$ inherits a dagger structure from $\TwoRep(G)$.
\end{definition}
\begin{remark}\label{rem:functionalchoice}
In basic quantum information theory, channels between matrix algebras are usually defined to be CP maps preserving the matrix trace, which in the covariant setting may not be $G$-invariant. For technical reasons related to duality of 1-morphisms in $\TwoRep(G)$~\cite[\S{}3.1.2]{Verdon2021} we have here stipulated that channels should preserve the separable standard $G$-invariant functional. This does not make a substantial difference to the overall theory. Indeed, had we chosen different $G$-invariant functionals the Frobenius algebra $A'$ corresponding to a given $G$-$C^*$-algebra would no longer be separable and standard, but it would be related to its separable and standard counterpart $A$ by an invertible $*$-homomorphism. Let $f_A: A' \to A$ and $f_B: B' \to B$ be such invertible $*$-homomorphisms; then there is a bijection between channels $\phi': A' \to B'$ and channels $\phi: A \to B$ given by $\phi' = f_B^{\dagger} \circ \phi \circ (f_A^{-1})^{\dagger}$.
\end{remark}
\noindent
In order to construct and study CP morphisms we now introduce the 2-category $\TwoRep(G)$. 

First observe that the 2-category $\TwoRep(G)$ has a privileged object, which we will write as $r_0$. (As discussed in Appendix~\ref{app:tworepdef}, if $\TwoRep(G)$ is defined in terms of $\F$s and dagger bimodules in $\Rep(G)$, this object is the trivial $\F$ $\mathbbm{1}$.) There is an equivalence $\End(r_0) \simeq \Rep(G)$.

Let $X: r_0 \to s$ be a separable 1-morphism out of the object $r_0$ in $\TwoRep(G)$. Then $X \otimes X^*$ is an object of $\End(r_0) \simeq \Rep(G)$ with the following \emph{multiplication} and \emph{counit} morphisms. (In the subsequent diagrams we leave the $r_0$-regions unshaded and shade the $s$-regions with wavy lines. Recall also that the identity 1-morphisms are invisible; we draw endomorphisms of identity 1-morphisms as `floating disks' surrounded by a dashed line.)
\begin{calign}\nonumber
&\includegraphics[scale=.6,valign=c]{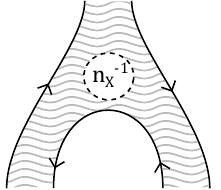}
&&
&\includegraphics[scale=.6,valign=c]{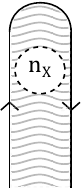}
\\\label{eq:pairofpants}
& m: (X \otimes X^*) \otimes (X \otimes X^*) \to X \otimes X^*
&&
& u^{\dagger}: X \otimes X^* \to \mathbbm{1}
\end{calign}
These morphisms give the object $X \otimes X^*$ the structure of a $\F$ in $\Rep(G)$. 
\begin{proposition}[{\cite[Cor. 3.37]{Chen2022}}] Up to isomorphism, all systems can be obtained from separable 1-morphisms $X: r_0 \to s$ out of the object $r_0$ in $\TwoRep(G)$, with the multiplication and counit defined in~\eqref{eq:pairofpants}.
\end{proposition} 
\noindent 
We call the algebra structure~\eqref{eq:pairofpants} on $X \otimes X^*$ a \emph{pair of pants} algebra, because of the appearance of the multiplication $m$.

The fact that all systems can be expressed as pair of pants algebras gives us an extremely useful description of CP morphisms between them. The following result is a finite-dimensional, $G$-covariant version of Stinespring's theorem~\cite{Stinespring1955}. In the diagrams we shade the $t$-regions with polka dots.
\begin{proposition}[{\cite[Thm. 4.11]{Verdon2021}}]\label{prop:dilation}
Let $X: r_0 \to s$ and $Y: r_0 \to t$ be separable 1-morphisms in $\TwoRep(G)$, let $X \otimes X^*$ and $Y \otimes Y^*$ be the corresponding systems, and let $f: X \otimes X^* \to Y \otimes Y^*$ be a 2-morphism. We say that $f$ \emph{admits a dilation} if there exists a 1-morphism $E: t \to s$ in $\TwoRep(G)$ (the \emph{environment}) and a 2-morphism $\tau: X \to Y \otimes E$ (the \emph{dilation}) such that $f$ is obtained as follows:
\begin{calign}\label{eq:stinespring}
\includegraphics[scale=.7,valign=c]{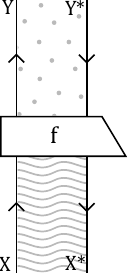}
~~=~~
\includegraphics[scale=.7,valign=c]{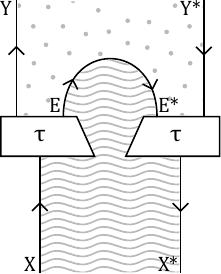}
\end{calign}
A 2-morphism $f: X \otimes X^* \to Y \otimes Y^*$ is a CP morphism iff it admits a dilation, and furthermore a channel iff it additionally preserves the counit, i.e.:
\begin{calign}\label{eq:channeldef}
\includegraphics[scale=.7,valign=c]{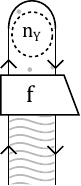}
~~=~~
\includegraphics[scale=.7,valign=c]{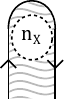}
\end{calign}
\end{proposition}
\noindent
Proposition~\ref{prop:dilation} allows us to study CP morphisms between systems in $\Rep(G)$ by studying their dilations in $\TwoRep(G)$. This may equivalently be seen as the covariant formulation of the Kraus representation for CP maps~\cite[\S{}8.2.3]{Nielsen2010}, which is foundational for quantum information theory; when $G$ is trivial we will show (in Example~\ref{ex:matrixqrel}) how to recover the Kraus maps from~\eqref{eq:stinespring}. One can think of Proposition~\ref{prop:dilation} as splitting a CP morphism open to reveal its internal structure, which by uniqueness of the minimal dilation (Lemma~\ref{lem:partialisom}) is uniquely defined. The utility of this approach will be demonstrated in what follows.
 
We give a couple of lemmas about dilations which we will use later on. 
\begin{lemma}[{\cite[Thm. 4.11]{Verdon2021}}]\label{lem:chancond}
A CP morphism $f: X \otimes X^* \to Y \otimes Y^*$ is a channel iff, for any dilation $\tau: X \to Y \otimes E$, the following morphism is an isometry:
\begin{calign}\label{eq:channelcond}
\includegraphics[scale=.7,valign=c]{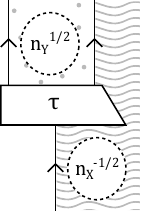}
\end{calign}
\end{lemma}
\begin{lemma}[{\cite[Thm. 4.11]{Verdon2021}}]\label{lem:partialisom}
Dilations are unique up to partial isometry on the environment. That is, if $f: X \otimes X^* \to Y \otimes Y^*$ is a CP morphism and $\tau_1: X \to Y \otimes E_1$ and $\tau_2: X \to Y \otimes E_2$ are dilations, there exists a partial isometry $\alpha: E_1 \to E_2$ such that:
\begin{align*}
(\id_Y \otimes \alpha) \circ \tau_1 = \tau_2 
&&
(\id_Y \otimes \alpha^{\dagger}) \circ \tau_2 = \tau_1
\end{align*}
In particular, the dilation $\tau_{\textrm{min}}: X \to Y \otimes E_{\textrm{min}}$ minimising the quantum dimension of the environment is unique up to a unitary $\alpha$ on the environment. The dilation is minimal iff the following morphism $Y^* \otimes X \to E$ possesses a right inverse:
\begin{calign}
\includegraphics[scale=.7,valign=c]{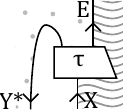}
\end{calign}
\end{lemma}
\noindent
We also have the following covariant generalisation of Choi's characterisation of completely positive maps~\cite[Thm. 2]{Choi1975}, which is fundamental to this work. 
\begin{theorem}[{Covariant Choi's theorem~\cite[Thm. 4.13]{Verdon2021}}]\label{thm:choi}
Let $X \otimes X^*$ and $Y \otimes Y^*$ be systems. Then there is a bijective correspondence (in fact, an isomorphism of convex cones, in the sense that it preserves positive linear combinations) between:
\begin{itemize}
\item CP morphisms $f: X \otimes X^* \to Y \otimes Y^*$.
\item Positive elements $\widetilde{f} \in \End(Y^* \otimes X)$.
\end{itemize}
The correspondence is given as follows:
\begin{calign}\label{eq:choi}
\includegraphics[scale=.7,valign=c]{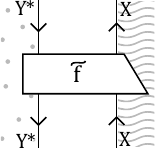}
~~~:=~~~
\includegraphics[scale=.7,valign=c]{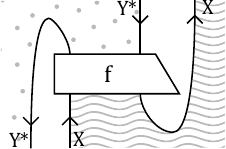}
\\\nonumber
\includegraphics[scale=.7,valign=c]{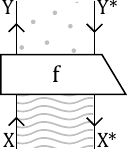}
~~~:=~~~
\includegraphics[scale=.7,valign=c]{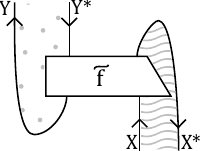}
\end{calign}
\end{theorem}
\begin{proof}
Let $\tau: X \to Y \otimes E$ be a dilation of $f$, then:
\begin{calign}
\includegraphics[scale=.7,valign=c]{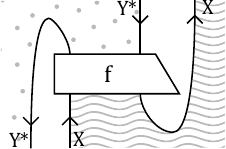}
~~=~~
\includegraphics[scale=.7,valign=c]{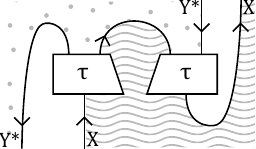}
~~=~~
\includegraphics[scale=.7,valign=c]{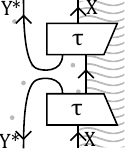}
\end{calign}
The last diagram is clearly the composition of a 2-morphism with its dagger and is therefore positive.

In the other direction, let $\tilde{f} \in \End(Y^* \otimes X)$ be positive. Then we can choose $m \in \End(Y^* \otimes X)$ such that $\tilde{f} = m^{\dagger} \circ m$, and transposing the relevant wires we obtain a dilation for $f$ with environment $Y^* \otimes X$:
\begin{align*}
\includegraphics[scale=.7,valign=c]{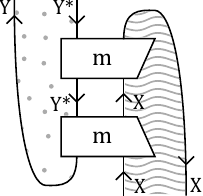}
~~=~~
\includegraphics[scale=.7,valign=c]{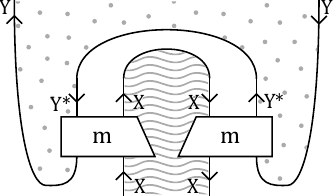}
\end{align*}
\end{proof}
\begin{notation}
Throughout this work we use the notation of Theorem~\ref{thm:choi}; that is, we write a positive element $\widetilde{f} \in \End(Y^* \otimes X)$ and its corresponding CP morphism $f: X \otimes X^* \to Y \otimes Y^*$ with the same latin letter, but use a tilde to indicate that we are referring to the positive element rather than the CP morphism.
\end{notation}
\noindent
Lastly, we consider some special examples of CP morphisms.
\begin{definition}
Let $X \otimes X^*$ and $Y \otimes Y^*$ be systems. A morphism $f: X \otimes X^* \to Y \otimes Y^*$ is called a \emph{$*$-homomorphism} if it obeys the following equations:
\begin{calign}\label{eq:hom}
\includegraphics[scale=.6,valign=c]{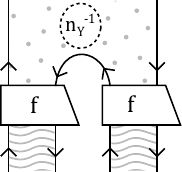}
~~=~~
\includegraphics[scale=.6,valign=c]{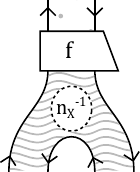}
&&
\includegraphics[scale=.6,valign=c]{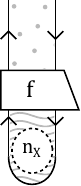}
~~=~~
\includegraphics[scale=.6,valign=c]{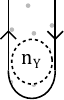}
&&
\includegraphics[scale=.6,valign=c]{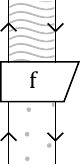}
~~=~~
\includegraphics[scale=.6,valign=c]{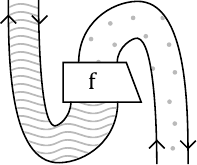}
\end{calign}
It is called a \emph{$*$-cohomomorphism} if it obeys the following equations:
\begin{calign}\label{eq:cohom}
\includegraphics[scale=.6,valign=c]{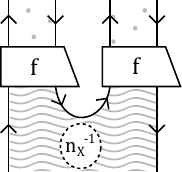}
~~=~~
\includegraphics[scale=.6,valign=c]{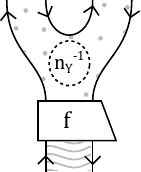}
&&
\includegraphics[scale=.6,valign=c]{pictures/covstinespring/starcohom21.pdf}
~~=~~
\includegraphics[scale=.6,valign=c]{pictures/covstinespring/starcohom22.pdf}
&&
\includegraphics[scale=.6,valign=c]{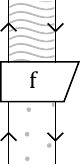}
~~=~~
\includegraphics[scale=.6,valign=c]{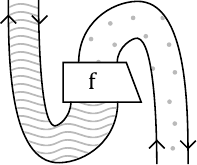}
\end{calign}
We will occasionally consider $*$-homomorphisms and $*$-cohomomorphisms which do not obey the second equations of~\eqref{eq:hom}\eqref{eq:cohom}. We call these \emph{non-unital} and \emph{non-counital} respectively. When we do not use these adjectives we mean that all equations are obeyed.
\end{definition}
\noindent
It is straightforward to show that $*$-homomorphisms and $*$-cohomomorphisms (even non-unital $*$-homomorphisms and non-counital $*$-cohomomorphisms) are CP morphisms. A $*$-cohomomorphism is additionally a channel (the second equation of~\eqref{eq:cohom} is precisely~\eqref{eq:channeldef}). The dagger of a $*$-homomorphism is a $*$-cohomomorphism, and vice versa.
\begin{remark}
As the name suggests, a $*$-homomorphism between systems is precisely a covariant $*$-homomorphism between the associated $G$-$C^*$-algebras. A $*$-cohomomorphism between systems is the Hermitian adjoint of a covariant $*$-homomorphism between the associated $G$-$C^*$-algebras, which are equipped with the inner product induced by the separable standard $G$-invariant functional. 
\end{remark}

\section{Covariant quantum relations and quantum $G$-graphs}\label{sec:relsandgraphs}

We can now define covariant quantum relations and quantum $G$-graphs using only categorical structure. We will motivate the definitions by applications to zero-error communication.

\subsection{Covariant quantum relations}

We will first define covariant quantum relations.  In classical information theory, a finite discrete memoryless channel $f$ with input alphabet $I$ and output alphabet $J$, written $f: I \to J$, is specified by a \emph{stochastic matrix} $(p_{ji})_{j \in J, i \in I}$; here $p_{ji}$ is the probability that the channel maps the input $i$ to the output $j$, and these probabilities satisfy $\sum_j p_{ji} = 1$. In the zero-error setting we are not interested in the magnitude of the transition probabilities, only in the possibility of transition. We therefore need only consider the relation $R \subset I \times J$ underlying the channel, defined by $R := \{(i,j)~|~p_{ji} \neq 0\} $.

Moving from a classical channel to its underlying relation induces a functor from the category $\Stoch$ of finite sets and stochastic matrices to the category $\Rel$ of finite sets and relations. We can generalise from stochastic matrices to matrices whose entries are positive real numbers; the same construction yields a functor $\Mat_{\mathbb{R}_{\geq 0}} \to \Rel$. The categories $\Mat_{\mathbb{R}_{\geq 0}}$ and $\Rel$ have daggers (given by the matrix transpose and the converse relation) such that this functor is unitary. Clearly, every relation arises from some morphism in $\Mat_{\mathbb{R}_{\geq 0}}$.

In order to quantise this construction, we first recall notions about supports and kernels for 2-morphisms in a semisimple $C^*$-2-category. 
\begin{definition}\label{def:suppann}
In any finite-dimensional $C^*$-algebra a partial order on projections is given by $p_1 \leq p_2$ iff $p_1p_2 = p_1 = p_2p_1$.

Let $X, Y: r \to s$ be 1-morphisms in $\TwoRep(G)$. For any $f: X \to Y$ the set $\mathcal{R}_f := \{a \in \End(X) ~|~ f \circ a = 0\}$ is a right ideal in $\End(X)$ which we call the \emph{right annihilator} of $f$; the \emph{left annihilator} $\mathcal{L}_f \subseteq \End(Y)$ is a left ideal defined by $\mathcal{L}_f:= \{a \in \End(Y) ~|~ a \circ f = 0\}$. By~\cite[Prop. 1.10.1]{Sakai2012} there are unique projections $e_{R,f} \in \End(X)$, $e_{L,f} \in \End(Y)$ such that $\mathcal{R}_f = e_{R,f} \circ \End(X)$ and $\mathcal{L}_f =  \End(Y) \circ e_{L,f}$. We call the projections $s_L(f):= \id_Y - e_{L,f}$ and $s_R(f) := \id_X - e_{R,f}$ the \emph{left support} and \emph{right support} of $f$, respectively. 

The projection $s_L(f)$ (resp. $s_R(f)$) may equivalently be defined as the least projection of all the projections $p \in \End(Y)$ (resp. $p \in \End(X))$ such that $p \circ f = f$ (resp. $f \circ p = f$). If $f$ is self-adjoint, i.e. $f^{\dagger} = f$, then $s_L(f) = s_R(f)$ and we call this projection the \emph{support} $s(f)$. We call the projection $\Ann(f):= 1-s(f)$ the \emph{annihilator}.
\end{definition}
\noindent
Recall the tilde notation for the Choi isomorphism (Theorem~\ref{thm:choi}).
\begin{definition}
A \emph{quantum relation} $p: X \otimes X^* \to Y \otimes Y^*$ is the CP morphism corresponding to a projection $\widetilde{p} \in \End(Y^* \otimes X)$. 

The \emph{underlying quantum relation} of a CP morphism $f: X \otimes X^* \to Y \otimes Y^*$ is defined by the projection $s(\tilde{f}) \in \End(Y^* \otimes X)$.

Let $p,q: X \otimes X^* \to Y \otimes Y^*$ be quantum relations; we say that $p \leq q$ if $\widetilde{p} \leq \widetilde{q}$.
\end{definition}
\noindent
Let us quickly show that we recover the standard notions of relation and quantum relation~\cite{Weaver2012,Kornell2020} when the group is trivial. In this case $\TwoRep(G)$ is just the category $\TwoFHilb$ described in~\cite[\S{}8]{Heunen2019}, and the graphical calculus reduces to a calculus for indexed families of linear maps which is fully summarised in~\cite[\S{}2.1]{Verdon2022}; we refer there for the details.
\begin{example}[Classical relations]\label{ex:classrel}
By Gelfand duality a finite set with $n$ elements corresponds to a commutative $C^*$-algebra with $n$ one-dimensional factors. In $\TwoHilb$ this is the system $X \otimes X^*$, where $X: [1] \to{} [n]$ is an indexed family of 1-dimensional Hilbert spaces $(X_{i})_{1 \leq i \leq n}$, $X_{i} \cong \mathbb{C} ~\forall~ j$. Let $X: [1] \to{} [n_1]$ and $Y: [1] \to{} [n_2]$ be two such 1-morphisms. Now let $f: X \otimes X^* \to Y \otimes Y^*$ be a CP morphism, and let $E: [n_2] \to{} [n_1]$ and $\tau: X \to Y \otimes E$ be a dilation of $f$. We can identify $f$ (on the left) with an indexed family of linear maps $X_i \otimes X_i^* \to Y_j \otimes Y_j^*$ (on the right):
\begin{calign}
\includegraphics[scale=.8,valign=c]{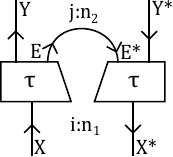}
\qquad \leftrightarrow \qquad 
\includegraphics[scale=.8,valign=c]{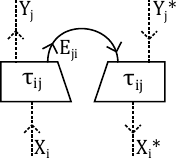}
~~=~~
\includegraphics[scale=.8,valign=c]{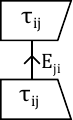}
\end{calign}
Note that when we move to the indexed family the 1-morphisms $X_{i}, (X_j)^*, Y_{j}$ and $(Y_j)^*$ disappear, since we do not depict the one-dimensional Hilbert space (i.e. the tensor unit) in the graphical calculus (we have drawn them using dotted lines on the left-hand side of the equality, but this is just for illustration). We see that each of the $i,j$-indexed morphisms is just a positive real number $f_{ij}:= \tau_{ij}^{\dagger} \circ \tau_{ij}$; the CP map $f$ therefore corresponds to a matrix of positive real numbers $(f_{ij})_{1 \leq i \leq m, 1 \leq j \leq n}$, namely a morphism in $\Mat_{\mathbb{R}_{\geq 0}}$. Now we consider the underlying relation. We observe that $\tilde{f} \in \End(Y^* \otimes X)$ is the following indexed family of linear maps $\mathbb{C} \to \mathbb{C}$:
\begin{calign}
\includegraphics[scale=.8,valign=c]{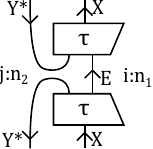}
\qquad \leftrightarrow \qquad 
\includegraphics[scale=.8,valign=c]{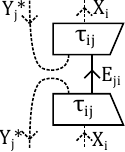}
~~=~~
\includegraphics[scale=.8,valign=c]{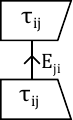}
\end{calign} 
The support of each of these maps is zero if the scalar is zero, and 1 otherwise. The relation $s(\tilde{f})$ therefore tells us which of the $f_{ij}$ are nonzero, as expected. 
\end{example}
\begin{example}[Quantum relations on matrix algebras]\label{ex:matrixqrel}
Let $X,Y: [1] \to{} [1]$ be 2-morphisms in $\TwoHilb$; since $\End([1]) \cong \Hilb$ these correspond to Hilbert spaces, and concretely the systems $X \otimes X^*$ and $Y \otimes Y^*$ are the matrix $C^*$-algebras $B(X)$ and $B(Y)$. 
A relation $B(X) \to B(Y)$ is therefore a projection $p \in \End(Y^* \otimes X)$; since $Y^* \otimes X \cong B(Y,X)$ (where $B(Y,X)$ is our notation for the linear maps $Y \to X$), this corresponds precisely to a subspace of $B(Y,X)$.

This subspace is usually defined in terms of the Kraus operators of the CP map. Let $f: X \otimes X^* \to Y \otimes Y^*$ be a CP morphism. Let $E: [1] \to{} [1]$ and $\tau: X \to Y \otimes E$ be a dilation of $f$. Now pick an orthonormal basis $\{\ket{i}\}$ for $E$ and consider the operators $M_i := (\id_Y \otimes \bra{i}) \circ \tau: X \to Y$. These are the Kraus operators associated with this dilation of $f$ and this choice of orthonormal basis. Now $\tilde{f}$ can be expressed as follows:
\begin{calign}
\includegraphics[scale=.8,valign=c]{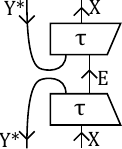}
~~=~~
\sum_i~
\includegraphics[scale=.8,valign=c]{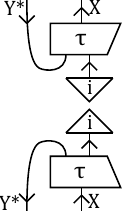}
~~=~~
\sum_i~
\includegraphics[scale=.8,valign=c]{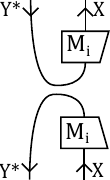}
\end{calign} 
Here for the first equality we inserted a resolution of the identity in terms of the basis $\{\ket{i}\}$. Looking at the final diagram we see that the support of this element of $\End(Y^* \otimes X)$ corresponds to the linear span of the Hermitian adjoints of the Kraus operators associated to the dilation.

(Given the description of quantum relations in the introduction it may be surprising to the reader that we obtained a subspace of $B(Y,X)$ rather than $B(X,Y)$. This is only a conventional issue; the other convention is to define the Choi isomorphism~\eqref{eq:choi} using the other partial transpose.)
\end{example}
\noindent
We now resume our treatment of covariant relations. Systems and quantum relations form a dagger category $\QRel(G)$, which has already been studied in some detail when the group action is trivial~\cite{Kornell2020}. In this category the composite of two quantum relations $p_1: X \otimes X^* \to Y \otimes Y^*$ and $p_2: Y \otimes Y^* \to Z \otimes Z^*$ is defined as the support $s(\widetilde{p}_2 \odot \widetilde{p}_1)$ of the following positive element $\widetilde{p}_2 \odot \widetilde{p}_1 \in \End(Z^* \otimes X)$:
\begin{calign}
\includegraphics[scale=.8,valign=c]{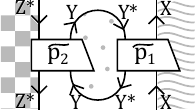}
\end{calign}
Here we have shaded the region corresponding to the target of $Z$ with a checkerboard shading. The identity relation $X \otimes X^* \to X \otimes X^*$, which in anticipation of the next section we call the \emph{discrete quantum confusability graph} $\Delta_{X\otimes X^*}$ on $X \otimes X^*$,  corresponds to the following projection $\widetilde{\Delta_{X \otimes X^*}} \in \End(X^* \otimes X)$: 
\begin{calign}\label{eq:discconfgraph}
\includegraphics[scale=.8,valign=c]{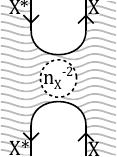}
\end{calign}
\noindent
The dagger $p^{\dagger}: Y \otimes Y^* \to X \otimes X^*$ of a relation $p: X \otimes X^* \to Y \otimes Y^*$ is defined by the following projection in $\End(X^* \otimes Y)$:
\begin{calign}\label{eq:convrel}
\includegraphics[scale=.8,valign=c]{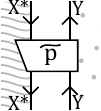}
\end{calign}
(In other words, the dagger of a quantum relation is given by the transpose of its corresponding projection.)
\begin{proposition}\label{prop:relfct}
Moving from a CP morphism to its underlying relation defines a full unitary (i.e. dagger-preserving) functor $\mathfrak{R}: \CP(G) \to \QRel(G)$.
\end{proposition}
\begin{proof}
Let $f: X \otimes X^* \to Y \otimes Y^*$, $g: Y \otimes Y^* \to Z \otimes Z^*$ be CP morphisms. We need to show that composition is preserved, i.e. that $\mathfrak{R}(g \circ f) = \mathfrak{R}(g) \circ \mathfrak{R}(f)$. This comes down to equality of the supports of the following positive elements $t_1, t_2 \in \End(Z^* \otimes X)$:
\begin{calign}
&\includegraphics[scale=.8,valign=c]{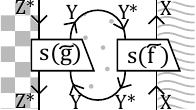}
&&
&\includegraphics[scale=.8,valign=c]{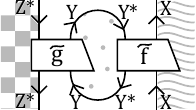}
\\
&t_1:= s(\widetilde{g}) \odot s(\widetilde{f})
&&
&t_2:= \widetilde{g} \odot \widetilde{f}
\end{calign}
By Definition~\ref{def:suppann}, equality of the supports can be rephrased in terms of the annihilators. Indeed, $s(t_1) = s(t_2)$ if and only if there is an equality of right annihilators $\mathcal{R}_{t_1} = \mathcal{R}_{t_2}$. 

Our argument depends on the following lemma. Let $\{f_i\}$ be a finite set of positive elements of an f.d.\ $C^*$-algebra $A$, and let $m_i \in A$ be such that $f_i = m_i^{\dagger} m_i$. Then we claim that, for any $a \in A$:
\begin{equation}\label{eq:positivesumkernel}
\sum_i f_i a = 0 \Leftrightarrow m_i a = 0~~\forall i
\end{equation} 
The leftwards implication is obvious. To show the rightwards implication we use the existence of a positive faithful trace $\Tr: A \to \mathbb{C}$. The derivation is as follows:
$$
\sum_i f_i a = 0 \Rightarrow \sum_i f_i a a^{\dagger} = 0 \Rightarrow \sum_i \Tr( a^{\dagger} f_i a) = 0 \Rightarrow  m_i a = 0~~\forall i
$$
Here the first implication is clear; the second implication is by traciality; and the final implication is by positivity and faithfulness of the trace.

Now we can show equality of the annihilators. First we show $\mathcal{R}_{t_1} \subseteq \mathcal{R}_{t_2}$: in other words, if $t_1 a = 0$ for some $a \in \End(Z^* \otimes X)$, then $t_2 a = 0$ also. We spectrally decompose $\widetilde{f}$ to obtain $\widetilde{f} = \sum_i \lambda_i p_i$, where $\{\lambda_i\}$ are positive and $p_i \in s(\widetilde{f}) \circ \End(Y^* \otimes X) \circ s(\widetilde{f})$ are orthogonal projections; we similarly obtain a decomposition $\widetilde{g} = \sum_j \sigma_j q_j$, where $\{\sigma_j\}$ are positive and $q_j \in s(\widetilde{g}) \circ \End(Z^* \otimes Y) \circ s(\widetilde{g})$ are orthogonal projections. It follows that $t_2 = \sum_{i,j} \lambda_i \sigma_j (q_j \odot p_i)$. By~\eqref{eq:positivesumkernel}, $t_2 a = 0$ iff $(q_j \odot p_i) a = 0$ for all $i,j$.  Now since $p_i \leq s(\widetilde{f})$ for all $i$, it follows that, for any $i$, $s(\widetilde{f}) = p_i + (s(\widetilde{f}) - p_i)$; here both of the summands are projections. Likewise, $s(\widetilde{g}) = q_j + (s(\widetilde{g}) - q_j)$ for any $j$. Then for any $i,j$ we can expand $t_1$ as a sum of positive elements:
\begin{align*}
t_1 = (q_j \odot p_i) + ((s(\widetilde{g}) - q_j) \odot p_i) 
+ (q_j \odot (s(\widetilde{f}) - p_i)) + ((s(\widetilde{g}) - q_j) \odot (s(\widetilde{f}) - p_i))
\end{align*}
So, by~\eqref{eq:positivesumkernel}, $t_1 a = 0$ implies that $(q_j \odot p_i) a = 0$ for all $i,j$; and therefore $t_2 a = 0$ also.

To see that $\mathcal{R}_{t_2} \subseteq \mathcal{R}_{t_1}$,  suppose that $t_2 a = 0$ for some $a \in \End(Z^* \otimes X)$. Since $\widetilde{f}$ and $\widetilde{g}$ are positive, we have $\widetilde{f} = m_1^{\dagger} \circ m_1$ and $\widetilde{g} = m_2^{\dagger} \circ m_2$ for some $m_1 \in \End(Y^* \otimes X)$ and $m_2 \in \End(Z^* \otimes Y)$. It follows from~\eqref{eq:positivesumkernel} that $(m_2 \nabla m_1) a = 0$, where $(m_2 \nabla m_1)$ is defined as follows:
\begin{calign}
\includegraphics[scale=.8,valign=c]{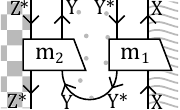}
\end{calign}
For a self-adjoint element $x$ in an f.d.\  $C^*$-algebra, the support  $s(x)$  is an element of the  commutative $C^*$-subalgebra generated by $x$~\cite[Prop. 1.10.4]{Sakai2012}. In particular, there is some finite $n \in \mathbb{N}$ such that $s(x) = \sum_{i=1}^{n} \lambda_i x^{i}$ for scalars $\{\lambda_i \in \mathbb{R}\}$. It follows that 
$t_1 = t (m_2 \nabla m_1)$ for some $t \in \End(Z^* \otimes X)$, and thus $t_1 a = 0$. We have therefore shown that $\mathfrak{R}(g \circ f) = \mathfrak{R}(g) \circ \mathfrak{R}(f)$. 

We must show that the identity morphisms are preserved. By definition, $\mathfrak{R}(\id_{X \otimes X^*})$ is the support of the following positive element $\widetilde{\id_{X \otimes X^*}} \in \End(X^* \otimes X)$:
\begin{calign}
\includegraphics[scale=.8,valign=c]{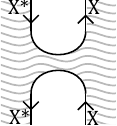}
\end{calign}
It is easy to see that this support is the projection $\widetilde{\Delta_{X \otimes X^*}}$ for the discrete confusability graph (by comparing the annihilators, for instance).

For unitarity we must show that $\mathfrak{R}(f^{\dagger}) = \mathfrak{R}(f)^{\dagger}$. We observe that the transpose of a CP morphism is equal to its dagger, i.e. $f^* = f^{\dagger}$:
\begin{calign}\label{eq:cpdaggertransp}
\includegraphics[scale=.8,valign=c]{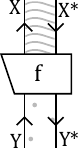}
~~=~~
\includegraphics[scale=.8,valign=c]{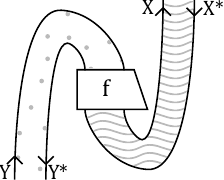}
~~=~~
\includegraphics[scale=.8,valign=c]{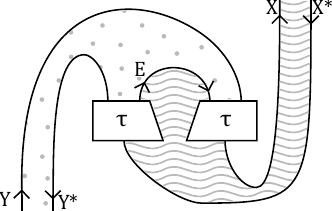}
~~=~~
\includegraphics[scale=.8,valign=c]{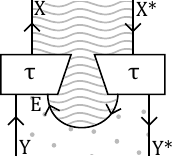}
~~=~~
\includegraphics[scale=.8,valign=c]{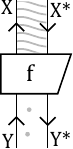}
\end{calign}
By definition, $\mathfrak{R}(f^{\dagger})$ is the support of the positive element $\widetilde{f^{\dagger}} \in \End(X^* \otimes Y)$ corresponding to $f^{\dagger}$ under Choi's theorem:
\begin{calign}
\includegraphics[scale=.8,valign=c]{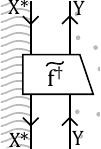}
~~=~~
\includegraphics[scale=.8,valign=c]{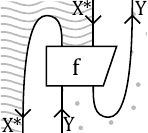}
~~=~~
\includegraphics[scale=.8,valign=c]{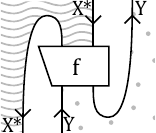}
~~=~~
\includegraphics[scale=.8,valign=c]{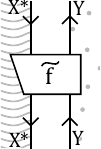}
\end{calign}
Here the second equality is by~\eqref{eq:cpdaggertransp}. In the last diagram we see $(\widetilde{f})^*$, whose support is clearly $s(\widetilde{f})^*$; but this was the definition of $\mathfrak{R}(f)^{\dagger}$~\eqref{eq:convrel}.

Finally, we must show fullness. A projector $\tilde{p} \in \End(Y^* \otimes X)$ is in particular positive, so corresponds to a CP map $p: X \otimes X^* \to Y \otimes Y^*$ under Choi's theorem~\eqref{eq:choi}; this CP map has $\tilde{p}$ as its underlying relation. 
\end{proof}
\noindent
\setstcolor{red}
We now provide a necessary \st{and sufficient} condition on a covariant quantum relation for it to be the underlying quantum relation of a covariant channel.
\begin{proposition}\label{prop:relchancond}
Let $p: X \otimes X^* \to Y \otimes Y^*$ be a relation. There is a channel $f: X \otimes X^* \to Y \otimes Y^*$ such that $\mathfrak{R}(f) = p$ \st{iff} \textcolor{red}{only if} the following positive element of $\End(X)$ is invertible:
\begin{calign}\label{eq:partialp}
\includegraphics[scale=.8,valign=c]{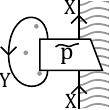}
\end{calign}
\end{proposition}
\begin{proof}
Suppose that there is a channel $f: X \otimes X^* \to Y \otimes Y^*$ such that $\mathfrak{R}(f) = p$. By~\eqref{eq:channeldef} and Choi's theorem~\eqref{eq:choi} we have the following equation:
\begin{align*}
\includegraphics[scale=.8,valign=c]{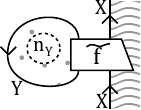}
~~=~~
\includegraphics[scale=.8,valign=c]{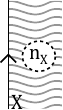}
\end{align*}
Noting that
\begin{align*}
s\left(
\includegraphics[scale=.8,valign=c]{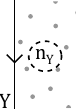}\right)
~=~
\id_{Y^*},
&&
s\left(\includegraphics[scale=.8,valign=c]{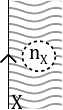}\right)
~=~
\id_{X},
&&
s(\widetilde{f}) = \widetilde{p},
\end{align*}
and using functoriality of $\mathfrak{R}$, we find that 
$$
s\left(\includegraphics[scale=.8,valign=c]{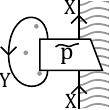}\right)
~=~
\includegraphics[scale=.8,valign=c]{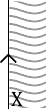}
$$
and so~\eqref{eq:partialp} is invertible (since it is a positive element with full support).

\st{In the other direction, suppose that}~\eqref{eq:partialp} \st{is invertible. This implies that the following element $t \in \End(X)$ is invertible:}
\begin{align*} 
\includegraphics[scale=.8,valign=c]{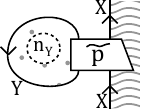}
\end{align*}
\st{(To see this, use functoriality of $\mathfrak{R}$ again.) 
We define the following positive element of $\End(Y^* \otimes X)$:}
\begin{align*}
\includegraphics[scale=.8,valign=c]{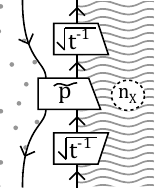}
\end{align*}
\st{Since this is conjugation of $\widetilde{p}$ by an invertible 2-morphism, its support is precisely $\widetilde{p}$ again. Since it is  a positive element of $\End(Y^* \otimes X)$, it corresponds by the Choi isomorphism}~\eqref{eq:choi} \st{to a CP map $X \otimes X^* \to Y \otimes Y^*$. By}~\eqref{eq:channeldef} \st{and}~\eqref{eq:choi}\st{, the following equation shows that this CP map is a channel:}
\begin{align*}
\includegraphics[scale=.8,valign=c]{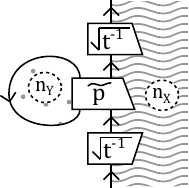}
~=~
\includegraphics[scale=.8,valign=c]{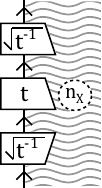}
~=~
\includegraphics[scale=.8,valign=c]{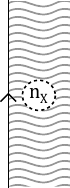}
\end{align*}
\end{proof}

\noindent
Before moving onto quantum graphs, we need to define (partial) quantum functions, which are a special sort of relation. The following definition and proposition are straightforward translations to rigid $C^*$-tensor categories of known results for hereditarily atomic von Neumann algebras.
\begin{definition}[{\cite[Def. 4.1]{Kornell2020}}]\label{def:function}
We say that a quantum relation $p: X \otimes X^* \to Y \otimes Y^*$ is a \emph{partial function} if it is \emph{coinjective}, i.e. $p \circ p^{\dagger} \leq \Delta_{Y\otimes Y^*}$. (Note that by $\circ$ here we mean the composition in $\QRel(G)$, not in $\Rep(G)$.) We say that a partial function is a \emph{function} if it is additionally \emph{cosurjective}, i.e. $\Delta_{X \otimes X^*} \leq p^{\dagger} \circ p$.
\end{definition}

\begin{proposition}[{c.f.~\cite[Thm. 6.3]{Kornell2020}}]\label{prop:partialfct}
Let $p: X \otimes X^* \to Y \otimes Y^*$ be a relation, and let $\iota: E \to Y^* \otimes X$ be an isometry which splits $\widetilde{p} \in \End(Y^* \otimes X)$, i.e. $\iota \circ \iota^{\dagger} = \widetilde{p}$. (Such an isometry always exists by local semisimplicity of $\TwoRep(G)$.) Then the following are equivalent:
\begin{enumerate}
\item The relation $p$ is a partial function.
\item The following 2-morphism is an isometry:
\begin{calign}\label{eq:partialfunctionisom}
\includegraphics[scale=.8]{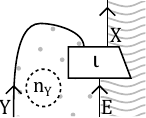}
\end{calign}
\item The following CP morphism $X \otimes X^* \to Y \otimes Y^*$ obeys the first and third equations of~\eqref{eq:cohom} (i.e. it is a possibly non-counital $*$-cohomomorphism):
\begin{calign}\label{eq:partialfunctioncp}
\includegraphics[scale=.8]{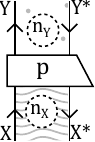}
\end{calign}
\end{enumerate}
Suppose now that $p$ is a partial function. Then the following are equivalent:
\begin{enumerate}
\item The partial function $p$ is a function. 
\item The following 2-morphism is an isometry:
\begin{calign}\label{eq:functionisom}
\includegraphics[scale=.8]{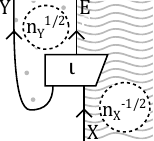}
\end{calign} 
\item The CP morphism~\eqref{eq:partialfunctioncp} is a $*$-cohomomorphism (i.e. it obeys all three equations of~\eqref{eq:cohom}). In particular, it is a channel.
\end{enumerate}
\end{proposition}

\begin{proof}
We show equivalences of the first group of statements. 
\begin{itemize}
\item (1 $\Rightarrow$ 2): Suppose that $p: X \otimes X^* \to Y \otimes Y^*$ is a partial function; this is to say that $s(\widetilde{p} \odot (\widetilde{p})^*) \leq \Delta_{Y \otimes Y^*}$. This implies that conjugation of $\widetilde{p} \odot (\widetilde{p})^*$ by $\Delta_{Y \otimes Y^*}$ preserves $\widetilde{p} \odot (\widetilde{p})^*$. We thereby obtain the following equation, where $x \geq 0$ is some positive element of the f.d.\ $C^*$-algebra $\End(\id_t)$:
\begin{calign}
\includegraphics[scale=.8,valign=c]{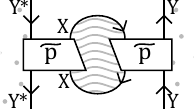}
~~=~~
\includegraphics[scale=.8,valign=c]{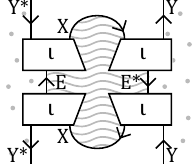}
~~=~~
\includegraphics[scale=.8,valign=c]{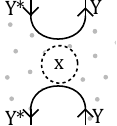}
\end{calign}
Here the first equality is by definition of the isometry $\iota$. Now by Choi's theorem~\eqref{eq:choi} we transpose the bottom left wire and the top right wire to move from an equation of positive elements of $\End(Y^* \otimes Y)$ to an equation of CP maps $Y \otimes Y^* \to Y \otimes Y^*$:
\begin{calign}\label{eq:partialfct1}
\includegraphics[scale=.8,valign=c]{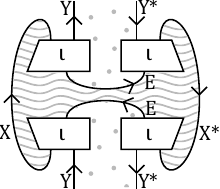}
~~=~~
\includegraphics[scale=.8,valign=c]{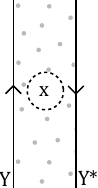}
\end{calign}
By uniqueness of a dilation up to partial isometry (Lemma~\ref{lem:partialisom}), and minimality of the dilation of the CP map on the RHS of~\eqref{eq:partialfct1}, there exists an isometry $\alpha: \id_t \to E \otimes E^*$ such that:
\begin{calign}
\includegraphics[scale=.8,valign=c]{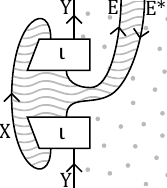}
~~=~~
\includegraphics[scale=.8,valign=c]{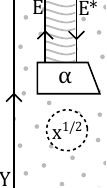}
\end{calign}
Now tracing out the $Y$-wire, we obtain the following equation for $\alpha$:
\begin{calign}\label{eq:partialfct2}
\includegraphics[scale=.8,valign=c]{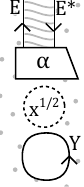}
~~=~~
\includegraphics[scale=.8,valign=c]{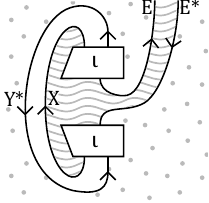}
~~=~~
\includegraphics[scale=.8,valign=c]{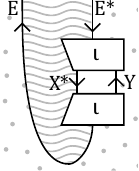}
~~=~~
\includegraphics[scale=.8,valign=c]{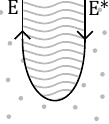}
\end{calign}
Transposing the $E^*$-wire, we thereby rewrite~\eqref{eq:partialfct2} as follows:
\begin{calign}
\includegraphics[scale=.8,valign=c]{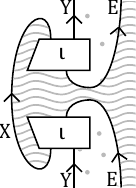}
~~=~~
\includegraphics[scale=.8,valign=c]{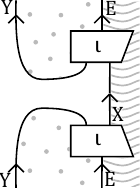}
~~=~~
\includegraphics[scale=.8,valign=c]{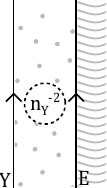}
\end{calign}
It follows immediately that~\eqref{eq:partialfunctionisom} is an isometry.
\item (2 $\Rightarrow$ 3):
We already know that the CP morphism $p$ obeys the third equation of~\eqref{eq:cohom}, because all CP morphisms do (we showed this in~\eqref{eq:cpdaggertransp}); so we only need to show that it obeys the first equation. 
Suppose that~\eqref{eq:partialfunctionisom} is an isometry. Then the first $*$-cohomomorphism equation is seen as follows:
\begin{calign}\nonumber
\includegraphics[scale=.8,valign=c]{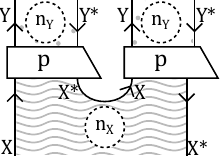}
~~=~~
\includegraphics[scale=.8,valign=c]{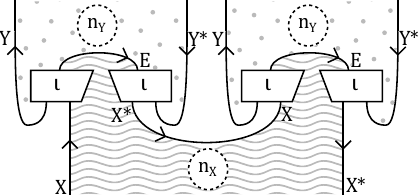}
\\
~~=~~
\includegraphics[scale=.8,valign=c]{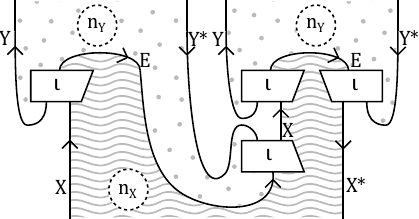}
~~=~~
\includegraphics[scale=.8,valign=c]{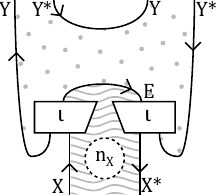}
\end{calign}
\item (3 $\Rightarrow$ 1): Consider the following equation for $\widetilde{p} \odot (\widetilde{p})^{*}$:
\begin{calign}\nonumber
\includegraphics[scale=.8,valign=c]{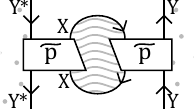}
~~=~~
\includegraphics[scale=.8,valign=c]{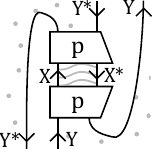}
~~=~~
\includegraphics[scale=.8,valign=c]{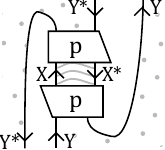}
\\
~~=~~
\includegraphics[scale=.8,valign=c]{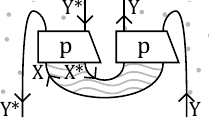}
~~=~~
\includegraphics[scale=.8,valign=c]{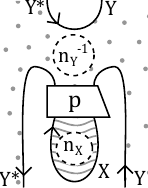}
~~=~~
\includegraphics[scale=.8,valign=c]{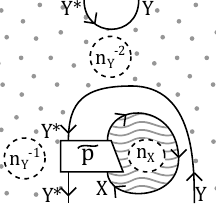}
\end{calign}
Here the first equality is by the Choi isomorphism~\eqref{eq:choi}; the second equality is by the third $*$-cohomomorphism condition~\eqref{eq:cohom}; the fourth equality is by the first $*$-cohomomorphism condition~\eqref{eq:cohom}; and the fifth equality is by the Choi isomorphism~\eqref{eq:choi}.
In the last diagram we see $\widetilde{p} \odot (\widetilde{p})^* = \Delta_{Y\otimes Y^*} \circ x$ for a positive $x \in \End(Y^* \otimes Y)$. It follows that $s(\tilde{p} \odot (\widetilde{p})^*) \leq \Delta_{Y\otimes Y^*}$.
\end{itemize}
We now show equivalence of the second group of statements.
\begin{itemize}
\item (1 $\Rightarrow$ 3): The fact that $\Delta_{X\otimes X^*} \leq s((\widetilde{p})^* \odot \widetilde{p})$ implies an inclusion of left annihilators $\mathcal{L}_{(\widetilde{p})^* \odot \widetilde{p}} \subseteq \mathcal{L}_{\widetilde{\Delta_{X\otimes X^*}}}$. We will now show that $1-t \in \mathcal{L}_{(\widetilde{p})^* \odot \widetilde{p}}$, where $t$ is defined as follows:
\begin{calign}
\includegraphics[scale=.8,valign=c]{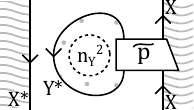}
\end{calign}
The fact that $1-t$ indeed annihilates $(\widetilde{p})^* \odot \widetilde{p}$ follows from the following equation:
\begin{calign}
\includegraphics[scale=.8,valign=c]{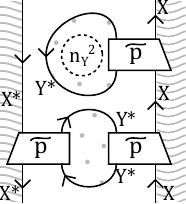}
~~=~~
\includegraphics[scale=.8,valign=c]{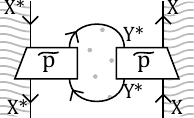}
\end{calign}
This equality follows from $\widetilde{p} = \iota \circ \iota^{\dagger}$ and~\eqref{eq:partialfunctionisom}. Now since $1-t$ annihilates $(\widetilde{p})^* \odot \widetilde{p}$, it must also annihilate $\widetilde{\Delta_{X \otimes X^*}}$ and so we have:
\begin{calign}\label{eq:function1}
\includegraphics[scale=.8,valign=c]{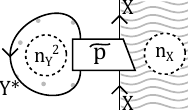}
~~=~~
\includegraphics[scale=.8,valign=c]{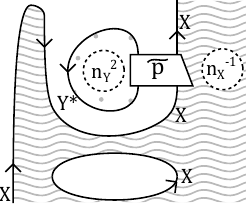}
~~=~~
\includegraphics[scale=.8,valign=c]{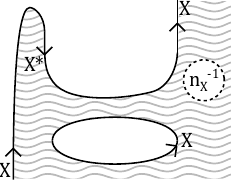}
~~=~~
\includegraphics[scale=.8,valign=c]{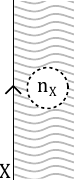}
\end{calign}
Moving from positive elements to CP morphisms using~\eqref{eq:choi}, the equation~\eqref{eq:function1} says precisely that:
\begin{calign}
\includegraphics[scale=.8,valign=c]{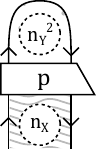}
~~=~~
\includegraphics[scale=.8,valign=c]{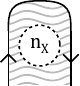}
\end{calign}
But this is precisely the second $*$-cohomomorphism condition~\eqref{eq:cohom} for~\eqref{eq:partialfunctioncp}.
\item (3 $\Leftrightarrow$ 2): This is precisely Lemma~\ref{lem:chancond}.
\item (3 $\Rightarrow$ 1): This will follow from Proposition~\ref{prop:confgraphfromchannel}, which shows that $s((\widetilde{p})^* \odot \widetilde{p}) \geq \widetilde{\Delta_{X \otimes X^*}}$.
\end{itemize}
\end{proof}

\subsection{Quantum $G$-graphs}

A covariant quantum relation encodes the `possibilistic' structure of its overlying covariant channel. In fact, for practical problems in zero-error communication, such as channel reversal and source channel coding, we will often need even less information than this.

Let us return to classical information theory for motivation. We began with a classical channel $f: I \to J$, i.e. a stochastic matrix $(p_{ji})_{j \in J, i \in I}$, and moved to its underlying relation $R: I \to J$. We now consider the relation $R^{\dagger} \circ R: I \to I$, where the dagger indicates the converse relation. Intuitively, two elements of $I$ are related by $R^{\dagger} \circ R$ if and only if they have a nonzero probability of being mapped to the same output under the channel $f$. In particular, the relation $R^{\dagger} \circ R$ is:
\begin{itemize}
\item Reflexive: $x \sim x$ for all $x \in I$.
\item Symmetric: $x \sim y \Leftrightarrow y \sim x$ for all $x,y \in I$.
\end{itemize}
This is a graph with vertex set $I$; we call it the \emph{confusability graph} of the channel $f$. It is not a simple graph, since every vertex is self-adjacent. However, its complement is a simple graph, which following~\cite{Stahlke2015} we call the \emph{distinguishability graph} of $f$. (We could do the same construction for a matrix in $\Mat_{\mathbb{R}_{\geq 0}}$ and we would obtain a graph, but this graph would not necessarily be reflexive; there might be some vertices with no adjoining edges.)

As before, we generalise these ideas to systems and channels in $\TwoRep(G)$.
\begin{definition}\label{def:confsimpgraphs}
We say that a covariant quantum relation $\Gamma: X \otimes X^* \to X \otimes X^*$ is \emph{symmetric} if $\Gamma^{\dagger} = \Gamma$ (or equivalently, if $(\widetilde{\Gamma})^* = \widetilde{\Gamma}$). We call a symmetric covariant quantum relation a \emph{quantum $G$-graph} on $X \otimes X^*$. We call $\Gamma \in \End(X \otimes X^*)$ the \emph{adjacency matrix} of the quantum $G$-graph.

We say that a quantum $G$-graph $\Gamma$ on $X \otimes X^*$ is a \emph{quantum confusability $G$-graph} (resp. a \emph{simple} quantum $G$-graph) if it obeys the left hand (resp. right hand) equation below:
\begin{calign}\nonumber
&\Delta_{X\otimes X^*} \leq \Gamma && & \widetilde{\Delta_{X \otimes X^*}} \circ \widetilde{\Gamma} = 0 
\\
&\textrm{confusability} && &\textrm{simple}
\end{calign}
\end{definition}
\noindent 
It is clear that $\id_{X^* \otimes X}-\widetilde{\Gamma}$ also defines a quantum $G$-graph on $X \otimes X^*$, which we call the \emph{complement} and write as $\Gamma^{\perp}$. The complement of a quantum confusability $G$-graph is a simple quantum $G$-graph and vice versa; this sets up a bijective correspondence between simple and confusability $G$-graphs. 
\begin{example}
We have already seen the \emph{discrete} quantum confusability $G$-graph $\Delta_{X\otimes X^*}$~\eqref{eq:discconfgraph}. 
The \emph{complete} quantum confusability $G$-graph $K_{X \otimes X^*}$ has the  projector $\id_{X^* \otimes X} \in \End(X^* \otimes X)$.

The \emph{discrete} simple quantum $G$-graph $K_{X \otimes X^*}^{\perp}$ has the projector $0 \in \End(X^* \otimes X)$. 
The \emph{complete} simple quantum $G$-graph $\Delta_{X \otimes X^*}^{\perp}$ has the projector $\id_{X^* \otimes X} - \widetilde{\Delta_{X\otimes X^*}} \in \End(X^* \otimes X)$.
\end{example}
\noindent
We now show that covariant channels give rise to confusability $G$-graphs, generalising the classical theory.
\begin{proposition}\label{prop:confgraphfromchannel}
Let $f: X \otimes X^* \to Y \otimes Y^*$ be a CP morphism. Then $\mathfrak{R}(f^{\dagger} \circ f) $ is a quantum $G$-graph on $X \otimes X^*$.
If $f$ is additionally a  channel, then it is a confusability $G$-graph.
\end{proposition}
\begin{proof}
Symmetry is immediate from $s((\tilde{f})^* \odot \tilde{f})^* = s(((\tilde{f})^* \odot \tilde{f})^*) = s((\tilde{f})^* \odot \tilde{f})$.

We now show that when $f$ is a channel we obtain a confusability $G$-graph. We want to show that $\widetilde{\Delta_{X\otimes X^*}} \leq s((\widetilde{f})^* \odot \widetilde{f})$. This is equivalent to showing an inclusion of left annihilators $\mathcal{L}_{(\widetilde{f})^* \odot \widetilde{f}} \subseteq \mathcal{L}_{\widetilde{\Delta_{X\otimes X^*}}}$. Let $a \in \mathcal{L}_{\tilde{f}^* \boxtimes \tilde{f}}$. We spectrally decompose $\tilde{f} = \sum_{i} \lambda_i p_i$, where $\lambda_{i} \in \mathbb{R}_{\geq 0}$. By~\eqref{eq:positivesumkernel} we have that
\begin{calign}\label{eq:pipjann}
\includegraphics[scale=.8,valign=c]{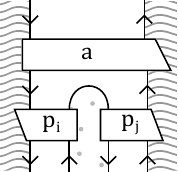} ~=~ 0~~~\forall~ i,j.
\end{calign}
Now we observe the following equation:
\begin{align*}
\includegraphics[scale=.8,valign=c]{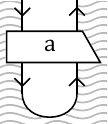}
~~=~~
\includegraphics[scale=.8,valign=c]{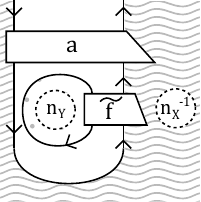}
~~=~~
\sum_j \lambda_j~
\includegraphics[scale=.8,valign=c]{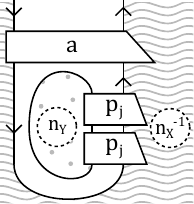}
~~=~~
\sum_j \lambda_j~
\includegraphics[scale=.8,valign=c]{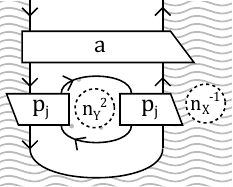}
\end{align*}
Here the first equality is by the counit preservation condition~\eqref{eq:channeldef}, the second equality is by spectral decomposition of $\tilde{f}$, and the third equality is by the sliding equations~\eqref{eq:sliding}. But then each of the summands in the last diagram is equal to zero by~\eqref{eq:pipjann}, and so $a \circ \Delta_{X \otimes X^*} = 0$.
\end{proof}
\noindent
We have seen that every channel has an associated quantum confusability $G$-graph, which might be expected. The less obvious fact is that every confusability $G$-graph arises from a covariant channel (thus justifying the name). The following proposition is a generalisation of~\cite[Lem. 2]{Duan2009}, which is the special case for quantum graphs on matrix algebras in the noncovariant setting.
\begin{proposition}\label{prop:graphtochan}
Let $\Gamma$ be a quantum confusability $G$-graph on a system $X \otimes X^*$. Then there exists an system $Y \otimes Y^*$ and a channel $f: X \otimes X^* \to Y \otimes Y^*$ such that $\Gamma = \mathfrak{R}(f^{\dagger} \circ f)$.
\end{proposition}
\begin{proof}
By definition of a quantum confusability graph we have that $\widetilde{\Delta_{X\otimes X^*}} \leq \widetilde{\Gamma}$. We can therefore decompose $\widetilde{\Gamma} = \widetilde{\Delta_{X\otimes X^*}} + (\widetilde{\Gamma} - \widetilde{\Delta_{X\otimes X^*}})$, where the summands are orthogonal projections. In fact the summands are quantum graphs, since we know that $\Gamma$ and $\Delta_{X\otimes X^*}$ are symmetric, and this implies that  $\widetilde{\Gamma} - \widetilde{\Delta_{X\otimes X^*}}$ is symmetric also. Moving via Choi's theorem from the positive operator to the corresponding CP morphism, we obtain $\Gamma = \Delta_{X\otimes X^*} + (\Gamma - \Delta_{X\otimes X^*})$, where the two summands are self-adjoint CP morphisms. 

\ignore{Here we define $\underline{\Delta_{X \otimes X^*}}$ to be the CP morphism such that $\underline{\widetilde{\Delta_{X\otimes X^*}}} = \widetilde{\Gamma} - \widetilde{\Delta_{X\otimes X^*}}$.}

By definition of the Choi isomorphism~\eqref{eq:choi} we have $\Delta_{X\otimes X^*} = \id_X \otimes (n_X)^{-2} \otimes \id_{X^*}$. This is a positive and invertible element of $\End(X \otimes X^*)$. It follows that $f:= \Delta_{X \otimes X^*} + \tau (\Gamma - \Delta_{X \otimes X^*})$ is positive, for small enough $\tau \in \mathbb{R}_{>0}$; we choose such a $\tau$. We now observe that $f$ is both positive and CP, which implies that $\widetilde{f} \in \End(X^* \otimes X)$ is also both positive and CP, in the sense that it admits a dilation~\eqref{eq:stinespring}. Moreover, it satisfies $s(\widetilde{f}) = s(\widetilde{\Delta_{X\otimes X^*}} + \tau (\widetilde{\Gamma} - \widetilde{\Delta_{X\otimes X^*}})) = s(\widetilde{\Gamma}) = \widetilde{\Gamma}$, since $\widetilde{\Delta_{X\otimes X^*}}$ and $ (\widetilde{\Gamma} - \widetilde{\Delta_{X\otimes X^*}})$ are orthogonal. \ignore{We also observe that $s(\tilde{f}^{\dagger}  \circ \tilde{f}) = s(\tilde{f}) = \widetilde{\Gamma}$.} Recall that $\eta_X^{\dagger}: X^* \otimes X \to \id_s$ is the notation for the cap of the left duality on $X$~\eqref{eq:cupscaps}; we observe that 
\begin{calign}
\label{eq:etaftilde}
\eta_X^{\dagger} \circ \widetilde{f} = \eta_X^{\dagger},
\end{calign}
since $\eta_X^{\dagger} \circ (\widetilde{\Gamma} - \widetilde{\Delta_{X\otimes X^*}}) = 0$ by orthogonality of $\widetilde{\Delta_{X\otimes X^*}}$ and $(\widetilde{\Gamma} - \widetilde{\Delta_{X\otimes X^*}})$. 

Let $\tau: X^* \to X^* \otimes E$ be a dilation of the CP morphism $\widetilde{f}$. (We remark that the left and right dimension of $E$ are always invertible, because it is an object of $\End(r_0) \simeq \Rep(G)$.)
We define the following CP morphism $g: X \otimes X^* \to E^* \otimes E$:
\begin{calign}
\includegraphics[scale=.8,valign=c]{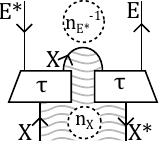}
\end{calign}
It is straightforward to show that $g$ is a channel:
\begin{calign}
\includegraphics[scale=.8,valign=c]{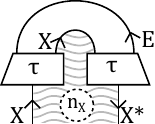}
~~=~~
\includegraphics[scale=.8,valign=c]{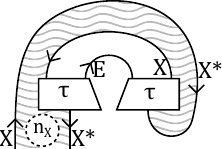}
~~=~~
\includegraphics[scale=.8,valign=c]{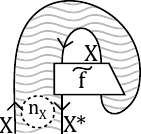}
~~=~~
\includegraphics[scale=.8,valign=c]{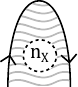}
\end{calign}
Here the first equality is by sliding~\eqref{eq:sliding}; we pulled the $\tau_*$ around the loop. The second equality is by the fact that $\tau$ is a dilation of $\widetilde{f}$. The third equality is by~\eqref{eq:etaftilde}.

We claim that this channel has $\Gamma$ as its confusability graph. Indeed, observe that the positive element $\widetilde{g^{\dagger} \circ g} \in \End(X^* \otimes X)$ is as follows:
\begin{calign}
\includegraphics[scale=.8,valign=c]{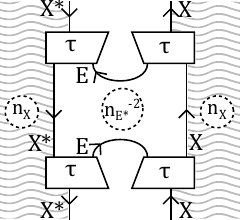}
\end{calign}
In inline notation, this is $\widetilde{f}^{\dagger} \circ (n_X \otimes \id_{X^*} \otimes n_{E^*}^{-2} \otimes \id_X \otimes n_X) \circ \widetilde{f}$. We have:
\begin{align*}
s(\widetilde{f}^{\dagger} \circ (n_X \otimes \id_{X^*} \otimes n_{E^*}^{-2} \otimes \id_X \otimes n_X) \circ \widetilde{f}) &= s_L( \widetilde{f}^{\dagger} \circ ((n_X)^{1/2} \otimes \id_{X^*} \otimes n_{E^*}^{-1} \otimes \id_X \otimes (n_X)^{1/2})) 
\\
&=s_L(\widetilde{f}^{\dagger}) = s(\widetilde{f}) = \widetilde{\Gamma} 
\end{align*}
Here the penultimate equality is by positivity of $\widetilde{f}$. We have therefore constructed a system $E^* \otimes E$ and a channel $g: X \otimes X^* \to E^* \otimes E$ such that $\mathfrak{R}(g^{\dagger} \circ g)$ = $\Gamma$.
\end{proof}
\noindent
We will finish the section  by defining covariant  homomorphisms of quantum $G$-graphs.
\begin{definition}\label{def:graphhoms}
Let $(A,\Gamma_A)$, $(B,\Gamma_B)$ be quantum confusability graphs. We say that a channel $f: A \to B$ satisfying
$$
\mathfrak{R}(f^{\dagger} \circ \Gamma_B \circ f) \leq \Gamma_A
$$
is  a \emph{homomorphism of confusability graphs} $(A,\Gamma_A) \to (B,\Gamma_B)$.

On the other hand, let $(A,\Gamma_A)$, $(B,\Gamma_B)$ be simple quantum graphs. We say that a channel $f: A \to B$ satisfying
$$
\mathfrak{R}(f \circ \Gamma_A \circ f^{\dagger}) \leq \Gamma_B
$$
is a \emph{homomorphism of simple graphs} $(A,\Gamma_A) \to (B,\Gamma_B)$.
\end{definition}
\begin{lemma}
Let $A, B$ be systems, let $\Gamma_A, \Gamma_B$ be confusability graphs on these systems, and let $\Gamma_A^{\perp}, \Gamma_B^{\perp}$ be their complementary simple graphs. A channel $f: A \to B$ is a homomorphism of confusability graphs $(A,\Gamma_A) \to (B,\Gamma_B)$ iff it is a homomorphism of simple graphs $(A,\Gamma_A^{\perp}) \to (B,\Gamma_B^{\perp})$. 
\end{lemma}
\begin{proof}
This is shown by the following sequence of equivalences, starting with the condition for $f$ to be a homomorphism of confusability graphs and finishing with the condition for $f$ to be a homomorphism of simple graphs:
\begin{align*}
\includegraphics[scale=.8,valign=c]{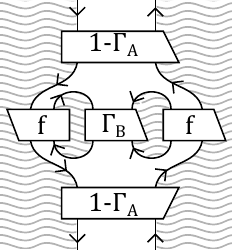}
~~=~~0
~~~\Leftrightarrow~~~
\Tr_{s}\left(~\includegraphics[scale=.8,valign=c]{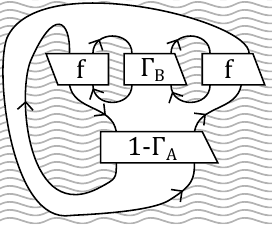}~\right)
=0
\end{align*}
\begin{align*}
\Leftrightarrow~~~
\includegraphics[scale=.8,valign=c]{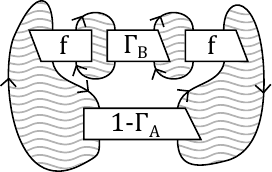}~~=~~0
~~~\Leftrightarrow~~~
\includegraphics[scale=.8,valign=c]{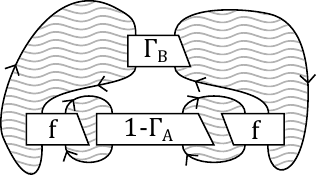}~~=~~0
\end{align*}
\begin{align*}
\Leftrightarrow~~~
\Tr_{s}\left(~\includegraphics[scale=.8,valign=c]{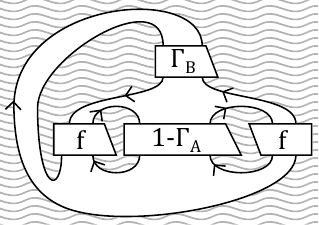}~\right)=0
~~~\Leftrightarrow~~~
\includegraphics[scale=.8,valign=c]{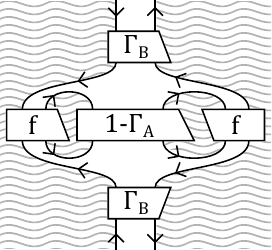}~~=~~0
\end{align*}
Here the first implication is by positivity and faithfulness of the trace~\eqref{eq:qtrace}, and isotopy (pull $(1-\Gamma_A$) around the loop, and use the fact that $(1-\Gamma_A)^2 = (1-\Gamma_A)^{\dagger} = (1-\Gamma_A)$ since it is a projection); the second is by the fact that the left trace is equal to the right trace~\eqref{eq:qtrace}; the third is by isotopy, pulling the $f$-boxes around the loop; the fourth is by the fact that the left trace is equal to the right trace; and the last is by positivity and faithfulness of the trace, and isotopy (again, use $\Gamma_B = (\Gamma_B)^2$ and pull one $\Gamma_B$ round the loop to get a positive operator in the trace).
\end{proof}
\begin{example}[Homomorphisms between classical graphs]\label{ex:stochgraphhom}
Let $A = [m]$ and $B = [n]$ be commutative $C^*$-algebras. Then simple graphs $(A,\Gamma_A)$ and $(B,\Gamma_B)$ are ordinary simple graphs with $m$ and $n$ vertices respectively (the CP maps $\Gamma_A, \Gamma_B$ are their adjacency matrices~\cite[Def. 5.1]{Musto2018}). A homomorphism $f: (A,\Gamma_A) \to (B,\Gamma_B)$ in the sense of Definition~\ref{def:graphhoms} is a stochastic matrix $(p_{ji})_{j \in [n],i \in [m]}$ whose underlying relation $\mathfrak{R}(f) \subset [m] \times [n]$ satisfies:
$$
(x,y) \in \mathfrak{R}(f) \Rightarrow \forall ~z \in [m]: (x \sim_{\Gamma_A} z \Rightarrow y \sim_{\Gamma_B} \tilde{z} ~~\forall~~ (z,\tilde{z}) \in \mathfrak{R}(f))
$$
\end{example}
\begin{remark}
This definition of homomorphism is suited to zero-error communication theory. However, since ordinary homomorphisms of classical graphs are functions, it would be natural to require that the channel defining a homomorphism should be a function, in the sense of Definition~\ref{def:function}. This produces a stronger definition of homomorphism which reduces to ordinary homomorphisms in the case of classical graphs, rather than the `stochastic' homomorphisms of Example~\ref{ex:stochgraphhom}. However, we need the weaker notion for zero-error source-channel coding. Indeed, it is easy to find quantum graphs between which there exist covariant homomorphisms in the sense of Definition~\ref{def:graphhoms}, but not in the stronger sense. For instance, consider the group $G=S_n$. Let $A = \mathbb{C}$ with the trivial action of $S_n$, and let $B = \mathbb{C}^{\oplus n}$ where $S_n$ permutes the factors. There is precisely one covariant channel $A \to B$, which is automatically a homomorphism for the complete confusability graphs on $A$ and $B$. However, there are no covariant functions from $A$ to $B$.
\end{remark}

\section{Reversibility of covariant channels}

First, following~\cite[\S{}2]{Duan2009}, we show that postprocessing of a channel only ever increases the size of the confusability graph. 
\begin{lemma}\label{lem:postprocessing}
Let $f: X \otimes X^* \to Y \otimes Y^*$ be a channel. Then for any channel $g: Y \otimes Y^* \to Z \otimes Z^*$, the confusability graph associated to $f$ is a subgraph of the confusability graph associated to $g \circ f$.
\end{lemma}
\begin{proof}
Since $g$ is a channel we know that $\Delta_{Y\otimes Y^*} \leq g^{\dagger} \circ g$. Thus:
\begin{align*}
s(\widetilde{f^{\dagger}} \odot (\widetilde{g^{\dagger}} \odot \widetilde{g}) \odot \widetilde{f}) &\geq s(\widetilde{f^{\dagger}} \odot \widetilde{\Delta_{Y\otimes Y^*}} \odot \widetilde{f})  
\\&=s(\widetilde{f^{\dagger}} \odot \widetilde{f})
\end{align*}
Here the equality is by functoriality of $\mathfrak{R}$ and the fact that $\Delta_{Y \otimes Y^*}$ is the identity on $Y$ in $\QRel(G)$.
\end{proof}
\noindent
We will now consider when a channel is reversible.
\begin{definition}
We say that a covariant channel $f: A \to B$ is \emph{reversible} if there exists another covariant channel $g: B \to A$ which is a left inverse for $f$, i.e. $g \circ f = \id_A$. 
\end{definition}
\noindent
It is clear that a classical channel $f$ is reversible if and only if its confusability graph is discrete. Indeed, this implies that the converse relation $\mathfrak{R}(f)^{\dagger}$ is a partial function, which can be extended to obtain a stochastic matrix which is a left inverse for $f$. This argument extends straightforwardly to general channels, as we now show. 

In the statement of the following lemma we will refer to the underlying relation of a channel as a confusability graph. By this we mean that it is a property of the underlying relation that it is a confusability graph, in the sense of Definition~\ref{def:confsimpgraphs}; we do not mean that it is the confusability graph of the channel. 
\begin{lemma}\label{lem:idchandisc}
There is precisely one channel $f: X \otimes X^* \to X \otimes X^*$ whose underlying relation is the discrete confusability graph on $X$, and this is the identity channel.
\end{lemma}
\begin{proof}
If $f$ is the identity channel then $\mathfrak{R}(f)$ is the discrete confusability graph by functoriality. In the other direction, we claim that if $\mathfrak{R}(f)$ is the discrete confusability graph then this implies that $\widetilde{f} \in \End(X^* \otimes X)$ has the following form, for some positive invertible $x \in \End(\id_s)$:
\begin{align*}
\includegraphics[scale=.8]{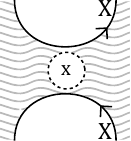}
\end{align*}
In inline notation, $\widetilde{f} = \eta_X \circ x \circ \eta_X^{\dagger}$, where $\eta_X$ is the cup of the duality on $X$. To see that the claim is true, observe that conjugating $\widetilde{f}$ by its support $\widetilde{\Delta_{X\otimes X^*}}$ we have $\widetilde{f} = \eta_X \circ (d_X^{-2} \otimes (\eta_X^{\dagger} \circ \widetilde{f} \circ \eta_X)) \circ  \eta_X^{\dagger}$; therefore let $x = d_X^{-2} \otimes (\eta_X^{\dagger} \circ \widetilde{f} \circ \eta_X)$. Clearly $x$ is positive. To see that it is invertible, take its support $s(x)$ and observe that $\pi := \eta_X \circ (d_X^{-1} \otimes s(x)) \circ \eta_X^{\dagger}$ is a projection satisfying $\pi \widetilde{f} = \widetilde{f}$, and $\pi \leq \widetilde{\Delta_{X\otimes X^*}} = s(\widetilde{f})$. By minimality of the support among such projections, we must have $\pi = \Delta_{X\otimes X^*}$, which implies $s(x) = 1$; therefore $x$ is invertible.

Finally, the trace preservation condition~\eqref{eq:channeldef} then implies that $n_X \circ x = n_X$, which by invertibility of $n_X$ gives $x = 1$.
\end{proof}
\begin{theorem}\label{thm:reversal}
A channel $f: X \otimes X^* \to Y \otimes Y^*$ is reversible if and only if its confusability graph is discrete. 
\end{theorem}
\begin{proof}
The `only if' direction is straightforward. Indeed, if $g \circ f = \id_{X \otimes X^*}$ then the confusability graph of $g \circ f$ is discrete. By Lemma~\ref{lem:postprocessing}, it follows that the confusability graph of $f$ is a subgraph of the discrete graph. However, it also contains the discrete graph, since $f$ is a channel. Therefore the confusability graph of $f$ must be precisely $\Delta_{X\otimes X^*}$.

For the `if' direction, let us suppose that the confusability graph of the channel $f$ is discrete. We will construct a left inverse for $f$. Since $\mathfrak{R}(f^{\dagger} \circ f) = \Delta_{X\otimes X^*}$, we have that $\mathfrak{R}(f^{\dagger}): Y \otimes Y^* \to X \otimes X^*$ is a partial function. Recall that $\widetilde{\mathfrak{R}(f^{\dagger})} = \widetilde{\mathfrak{R}(f)}^*$. We will now show that the following element of $\End(Y)$ is a projection:
\begin{calign}\label{eq:reversalprojdef}
\includegraphics[scale=.8]{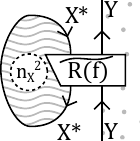}
\end{calign}
It is clearly positive. Let $\iota: E \to X^* \otimes Y$ be an isometry splitting the projection $\widetilde{\mathfrak{R}(f^{\dagger})}$, i.e. $\iota \circ \iota^{\dagger} = \widetilde{\mathfrak{R}(f^{\dagger})}$ (this always exists by local semisimplicity of $\TwoRep(G)$); then idempotency is seen as follows:
\begin{calign}
\includegraphics[scale=.8,valign=c]{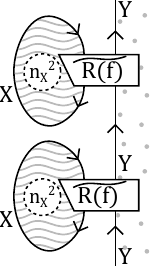}
~~=~~
\includegraphics[scale=.8,valign=c]{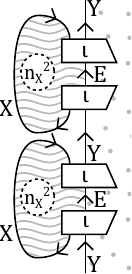}
~~=~~
\includegraphics[scale=.8,valign=c]{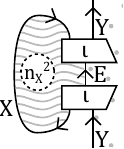}
~~=~~
\includegraphics[scale=.8,valign=c]{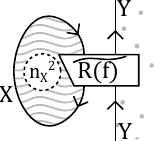}
\end{calign} 
Here for the second equality we used the characterisation~\eqref{eq:partialfunctionisom} of a partial function in terms of the isometry $\iota$.

We define the following positive element $\widetilde{g} \in \End(X^* \otimes Y)$:
\begin{calign}
\includegraphics[scale=.8,valign=c]{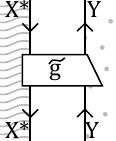}~
~:=~
\includegraphics[scale=.8,valign=c]{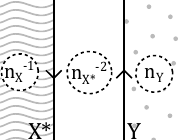}~
-~
\includegraphics[scale=.8,valign=c]{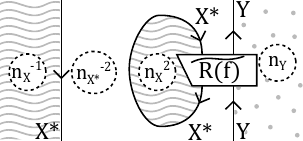}
~+~
\includegraphics[scale=.8,valign=c]{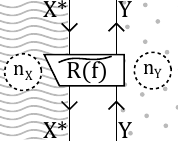}
\end{calign}
This is a positive element because the first two terms are a projection $\id_{X^*} \otimes \id_Y - \id_{X^*} \otimes \alpha$ (where $\alpha \in \End(Y)$ is the projection~\eqref{eq:reversalprojdef}) conjugated by $n_X^{-1/2} \otimes \id_{X^*} \otimes n_{X^*}^{-1} \otimes \id_Y \otimes n_Y^{1/2}$, and the last term is clearly positive. 
The following equation shows that the corresponding CP morphism $g: Y \otimes Y^* \to X \otimes X^*$ is a channel~\eqref{eq:channeldef}:
\begin{align*}
\includegraphics[scale=.8,valign=c]{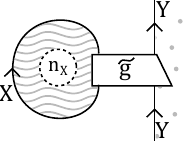}
~~&=~~
\includegraphics[scale=.8,valign=c]{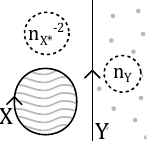}
~-~
\includegraphics[scale=.8,valign=c]{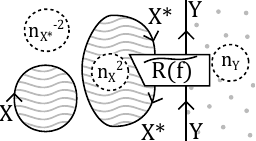}
~+~
\includegraphics[scale=.8,valign=c]{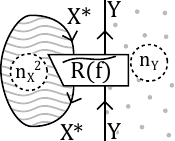}
\\
&=~~
\includegraphics[scale=.8,valign=c]{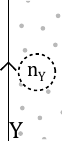}
~-~
\includegraphics[scale=.8,valign=c]{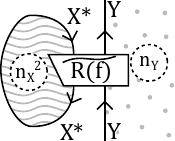}
~+~
\includegraphics[scale=.8,valign=c]{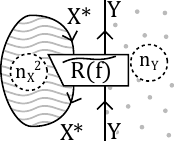}
\\
&=~~
\includegraphics[scale=.8,valign=c]{pictures/zeroerror/reversalpf21.pdf}
\end{align*}
Using functoriality of $\mathfrak{R}$, we can conjugate $\widetilde{\mathfrak{R}(f^{\dagger})} \odot \widetilde{f}$ by its support $\widetilde{\Delta_{X \otimes X^*}}$ to obtain the following equation:
\begin{calign}\nonumber
\includegraphics[scale=.8,valign=c]{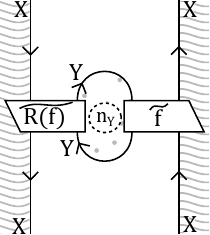}
~~=~~
\includegraphics[scale=.8,valign=c]{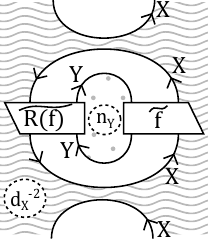}
~~=~~
\includegraphics[scale=.8,valign=c]{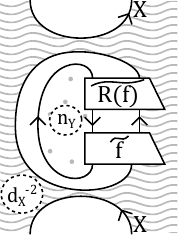}
\\\label{eq:reversal1}
=~~
\includegraphics[scale=.8,valign=c]{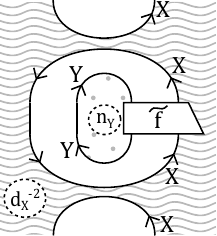}
~~=~~
\includegraphics[scale=.8,valign=c]{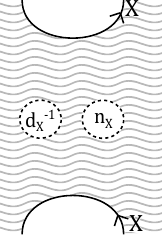}
\end{calign}
Here the second equality is by the sliding equations~\eqref{eq:sliding}; the third equality is by the fact that $\widetilde{\mathfrak{R}(f)}$ is the support of $\widetilde{f}$ by definition; and the fourth equality is by the fact that $f$ is a channel~\eqref{eq:channeldef}.

Using~\eqref{eq:reversal1}, we can show that $g$ is a left inverse for $f$:
\begin{align*}
\includegraphics[scale=.8,valign=c]{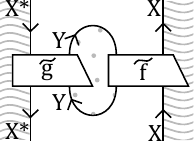}
~~&=~~
\includegraphics[scale=.8,valign=c]{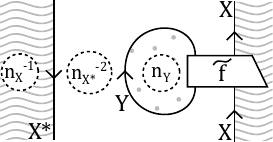}
~-~
\includegraphics[scale=.8,valign=c]{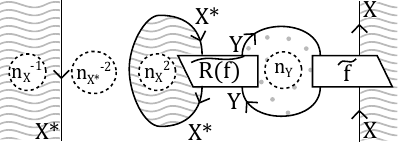}
\\&~~~~~~~~~~~~
~+~
\includegraphics[scale=.8,valign=c]{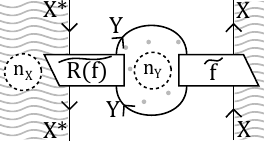}
\\
&=~~
\includegraphics[scale=.8,valign=c]{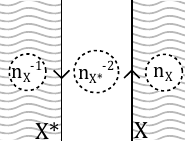}
~-~
\includegraphics[scale=.8,valign=c]{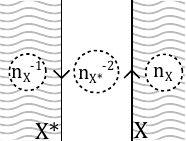}
~+~
\includegraphics[scale=.8,valign=c]{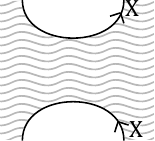}
\\
&=~~
\includegraphics[scale=.8,valign=c]{pictures/zeroerror/reversalpf43.pdf}
\end{align*}
Here the second equality uses~\eqref{eq:channeldef} and~\eqref{eq:reversal1}. The support of the final diagram is the discrete confusability graph, so since $g \circ f$ is a channel it must be the identity channel by Lemma~\ref{lem:idchandisc}. 
\end{proof}

\section{Covariant zero-error source-channel coding}\label{sec:scc}

We finish by considering zero-error source-channel coding. A long list of operational scenarios which can be formulated as quantum zero-error source-channel coding problems was given in~\cite[\S{}IV]{Stahlke2015}. 

In order to define a tensor product of arbitrary systems, we will here restrict to the case where $G$ is a quasitriangular compact quantum group. (See Appendix~\ref{app:monoidal} for the definition of the tensor product.) We recall the definition of zero-error source-channel coding from the introduction.

\begin{definition}[Covariant source channel coding]
Alice and Bob share a covariant communication channel $N: A \to B$. Charlie wants to send the state of a system $S$ to Bob. To do this, he will transmit information to Alice (a state of a system $O_A$) and some `side information' to Bob (a state of a system $O_B$). This transmission is defined by a covariant channel $C: S \to O_A \otimes O_B$. This data defines the zero-error source-channel coding problem.

The problem is as follows: Alice must use the channel $N$ to transmit information to Bob in order that Bob can recover the original state of the system $S$. Such a procedure is defined by an covariant encoding channel $E: O_A \to A$ (performed by Alice) and a covariant decoding channel $D: B \otimes O_B \to S$ (performed by Bob), and is called a \emph{covariant zero-error source-channel coding scheme}.

Mathematically, a valid choice of $E$ and $D$ is specified as follows. Let $A \cong X \otimes X^*$, $B \cong Y \otimes Y^*$, $S \cong Z_S \otimes Z_S^*$, $O_A \cong Z_A \otimes Z_A^*$ and $O_B \cong Z_B \otimes Z_B^*$  be the systems involved; here $X: r_0 \to x$, $Y: r_0 \to y$, $Z_S: r_0 \to s$, $Z_A: r_0 \to a$ and $Z_B: r_0 \to b$ are 1-morphisms in $\TwoRep(G)$.

 Let $\tau: X \to Y \otimes E_N$ be a dilation of the channel $N$, let $\beta: Z_S \to (Z_A \boxtimes Z_B) \otimes E_C$ be a dilation of the channel $C$ (here $E_C$ has type $a \boxtimes b \to s$) , let $\epsilon: Z_A \to X \otimes E_{E}$ be a dilation of the encoding channel $E$, and let $\delta: B \boxtimes Z_B \to Z_S \otimes E_D$ be a dilation of the decoding channel $D$ (here $E_D$ has type $s \to y \boxtimes b$). Then we require the following equation:
\begin{calign}\label{eq:scceq}
\includegraphics[scale=.8,valign=c]{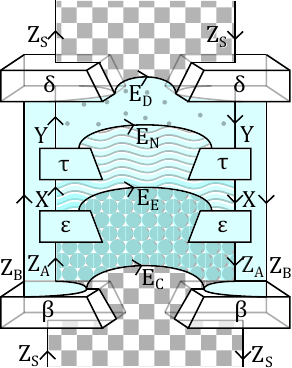}
~~=~~
\includegraphics[scale=.8,valign=c]{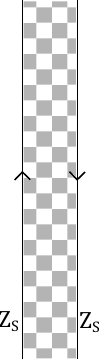}
\end{calign}
Here we have used the 3-dimensional graphical calculus for monoidal 2-categories. We have coloured the $s$-region with chequerboard shading, the $a$-region with packed circles, the $x$-region with wavy lines, the $y$-region with polka dots, and the $b$-region with a translucent blue colour. Note that in the left hand diagram, in between the $\beta$ and the $\delta$-levels there are two planes; a $b$-plane in the foreground between the $Z_B$ wires, and another plane in the background on which the $\tau$ and $\epsilon$ 2-morphisms are located.

Note that although we have drawn the $\beta$ and $\delta$ boxes as 3-dimensional objects in order to represent the fact that they map to or from a tensor product of 1-morphisms, they are really just the same 2-morphism boxes as we have been dealing with throughout this work. Every 2-morphism has a dagger, a transpose and a conjugate defined by the dagger structure and the duality on 1-morphisms; these are represented by reorienting the box, as in~\eqref{eq:boxesdefine}. For the 3-dimensional boxes we have represented the dagger, transpose and conjugate by reorienting the 3-dimensional box in the same way. For example, the box on the bottom left of the LHS of~\eqref{eq:scceq} represents $\beta$, and the box on the bottom right of that diagram represents the conjugate $\beta_*$. 
\end{definition}
\noindent
We will now show that the encoding channel $E$ in a valid source channel coding scheme is precisely a homomorphism between two confusability graphs. The target graph of this homomorphism is the confusability graph of the channel $N$. The source graph of this homomorphism is the \emph{confusability graph of the source $C$}, which we will now define. 

First observe that $C$ must be a reversible channel, or it would obviously be impossible to define a zero-error source-channel coding scheme. Its confusability graph is therefore discrete by Theorem~\ref{thm:reversal}. 

Suppose that some encoding channel $E$ is given. Let $Q$ be the confusability graph of the channel $N \circ E$. For a decoding channel $D$ to exist satisfying~\eqref{eq:scceq}, the channel $((N \circ E) \otimes \id_{O_B}) \circ C$ must be reversible. By Theorem~\ref{thm:reversal} and Proposition~\ref{prop:relfct}, that is precisely to say that the support of the following positive element of $\End(Z_S^* \otimes Z_s)$ is the projection corresponding to the discrete confusability graph~\eqref{eq:discconfgraph}:
\begin{calign}\label{eq:discrrequirement}
\includegraphics[scale=.8,valign=c]{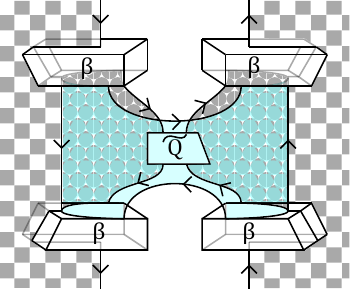}
\end{calign}
\noindent
(Note that, again, there are two planes in the centre of the diagram; the blue $b$-plane is in the foreground, and the $a$ plane with packed circles, with which the 2-morphism $Q$ interacts, is in the background. Again, we have used 3-dimensional 2-morphism boxes and represented the dagger, transpose and conjugate by reorienting the box, so the box on the bottom right represents $\beta$, the box on the bottom left its conjugate, the box on the top left its transpose, and the box on the top right its dagger.) 

This is precisely to say that the annihilator of~\eqref{eq:discrrequirement} is the projection $\widetilde{\Delta_{X \otimes X^*}^{\perp}} = 1-\widetilde{\Delta_{X \otimes X^*}}$ of the complete simple quantum graph. We then have the following series of implications:
\begin{calign}\nonumber
\includegraphics[scale=.7,valign=c]{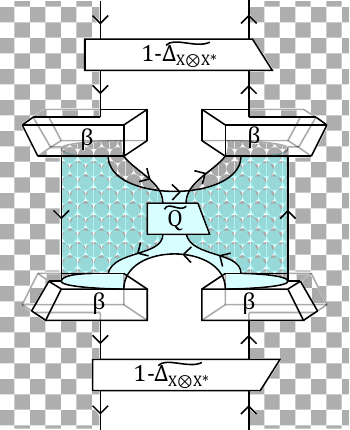}= 0
~~~\Leftrightarrow~~~
\Tr_s\left(~\includegraphics[scale=.7,valign=c]{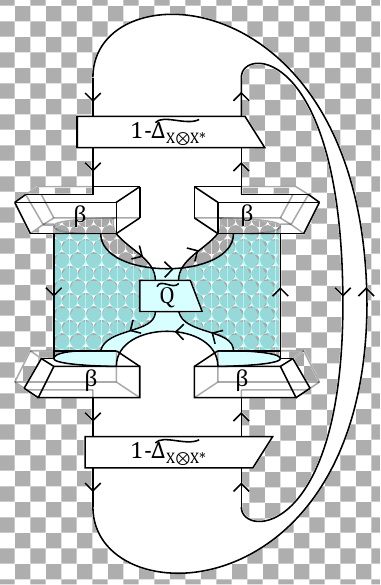}~\right)=0
\end{calign}
\begin{calign}\nonumber
\Leftrightarrow~~~\Tr_s\left(~\includegraphics[scale=.7,valign=c]{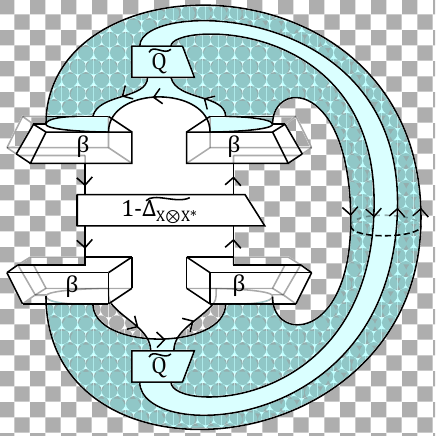}~\right)= 0
\end{calign}
\begin{calign}\nonumber
\Leftrightarrow~~~
\Tr_{a \boxtimes b}\left(~\includegraphics[scale=.7,valign=c]{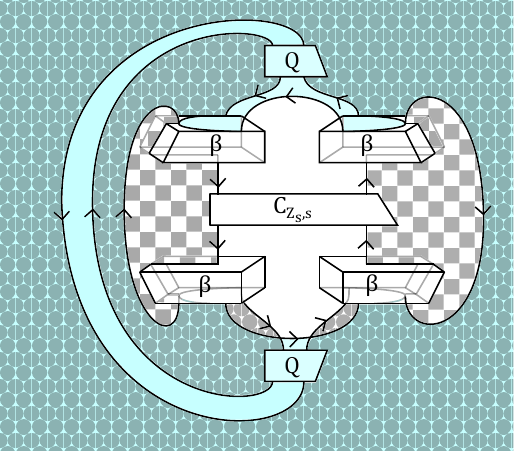}~\right) = 0
\end{calign}
\begin{calign}\nonumber
\Leftrightarrow~~~
\Tr_b\left(~\includegraphics[scale=.7,valign=c]{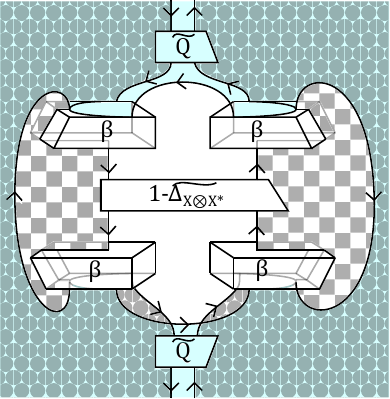}~\right) = 0
\end{calign}
Here the first equivalence is by positivity and faithfulness of the trace~\eqref{eq:qtrace}. The second equivalence is by using $\widetilde{Q} = \widetilde{Q}^{\dagger} \circ \widetilde{Q}$; pulling $\widetilde{\Delta_{X \otimes X^*}^{\perp}}$, $\beta^{*} \otimes \beta^{\dagger}$ and $\widetilde{Q}^{\dagger}$ around the loop; and using $(\widetilde{\Delta_{X \otimes X^*}^{\perp}})^{\dagger} \circ \widetilde{\Delta_{X \otimes X^*}^{\perp}} = \widetilde{\Delta_{X \otimes X^*}^{\perp}}$. The third equivalence is by the fact that the left trace is equal to the right trace~\eqref{eq:qtrace}. The fourth equivalence is by positivity and faithfulness of the traces $\Tr_{a \boxtimes b}$ and $\Tr_a$ and the partial trace $\Tr_b$, and the fact that $\Tr_{a \boxtimes b} = \Tr_a \circ \Tr_b$ (Lemma~\ref{lem:partialtrace}). 

The last equation says precisely that $\widetilde{Q} \leq \Ann(\widetilde{f})$, where $\widetilde{f} \in \End(Z_A^* \otimes Z_A)$ is the following positive element:
\begin{calign}\label{eq:ftilde}
\Tr_b\left(~\includegraphics[scale=.8,valign=c]{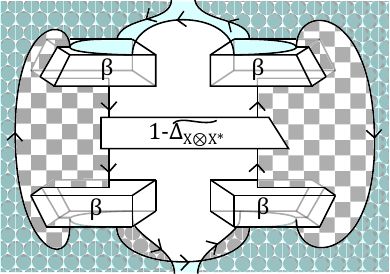}~\right)
\end{calign} 
We will show that $s(\widetilde{f})$ is a projection defining a simple quantum graph, which implies that $\Ann(\widetilde{f})$ is a projection defining a quantum confusability graph. It is clear that $s(\tilde{f})$ is symmetric, since $\widetilde{f}$ is symmetric; we therefore need only show that $s(\widetilde{f})$ is orthogonal to the discrete confusability graph, i.e. 
\begin{calign}\label{eq:confgraphorthogeq}
\widetilde{\Delta_{Z_A \otimes Z_A^*}} \circ \tilde{f} \circ \widetilde{\Delta_{Z_A \otimes Z_A^*}} = 0
\end{calign} 
which is seen as follows (here we have used positivity and faithfulness of the trace $\Tr_a$ to reduce~\eqref{eq:confgraphorthogeq} to the first expression being equal to zero):
\begin{calign}\nonumber
\Tr_{a \boxtimes b}\left(~\includegraphics[scale=.8,valign=c]{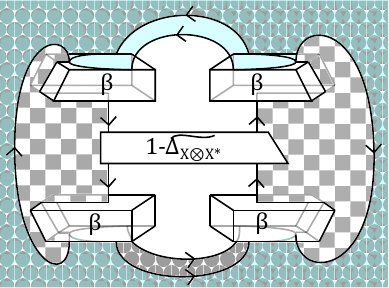}~\right)
~~=~~
\Tr_{s}\left(~\includegraphics[scale=.8,valign=c]{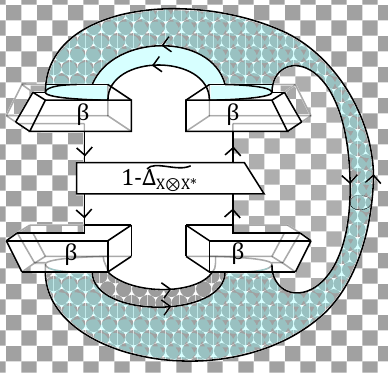}~\right)
\\\nonumber
=~~
\Tr_{s}\left(~\includegraphics[scale=.8,valign=c]{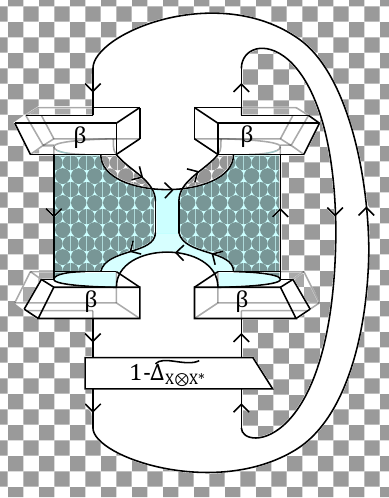}~\right)
~~=~~
\Tr_{s}\left(~\includegraphics[scale=.8,valign=c]{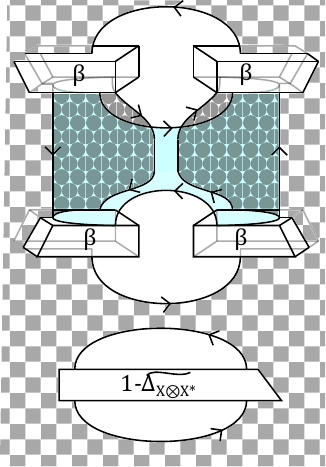}~\right)
~~=~~0
\end{calign}
Here the first equality is by equality of the left and right trace~\eqref{eq:qtrace}; the second equality is by pulling $\beta^* \otimes \beta^{\dagger}$ around the loop; the third equality is by discreteness of the confusability graph of the channel $C$, i.e. $\Delta_{Z_S \otimes Z_S^*}\circ \mathfrak{R}(C^{\dagger} \circ C) \circ \Delta_{Z_S \otimes Z_S^*} = \mathfrak{R}(C^{\dagger} \circ C)$; and the fourth equality is by simplicity of the graph $\Delta_{Z_S \otimes Z_S^*}$.

We therefore make the following definition.
\begin{definition}\label{def:confsourcegraph}
Let $C: S \to O_A \otimes O_B$ be a channel, and let $\beta: Z_S \to (Z_A \boxtimes Z_B) \otimes E$ be a dilation. Then we define the \emph{confusability graph of the source $C$}, written $(O_A,\Gamma_C)$, to be the confusability graph on $O_A$ given by the projection 
$$
\widetilde{\Gamma_C}:=(1-s(\widetilde{f})) \in \End(Z_A^* \otimes Z_A)
$$
where $\widetilde{f} \in \End(Z_A^* \otimes Z_A)$ is the positive element shown in~\eqref{eq:ftilde}.
\end{definition}
\begin{theorem}\label{thm:scchoms}
A channel $E: O_A \to A$ is an encoding channel for a source channel coding problem defined by a source $C:S \to O_A \otimes O_B$ and a communication channel $N: A \to B$ precisely when it is a graph homomorphism from the confusability graph of the source $(O_A,\Gamma_C)$ to the confusability graph of the channel $(A,\Gamma_N)$.
\end{theorem}
\begin{proof}
By Definition~\ref{def:graphhoms} the channel $E$ is a graph homomorphism precisely when $\mathfrak{R}(E^{\dagger} \circ \Gamma_N \circ E) \leq \Gamma_C$. But $\widetilde{\mathfrak{R}}(E^{\dagger} \circ \Gamma_N \circ E) = \widetilde{Q}$, and we saw above that $E$ is a valid encoding channel precisely when $\widetilde{Q} \leq \Ann(\widetilde{f}) = \widetilde{\Gamma_C}$.
\end{proof}
\noindent
We will now show that every graph is the confusability graph of some source, which together with Proposition~\ref{prop:graphtochan} and Theorem~\ref{thm:scchoms} gives an operational semantics for covariant homomorphisms of quantum confusability $G$-graphs (or equivalently, simple quantum $G$-graphs).
\begin{proposition}[{c.f. \cite[Thm. 17]{Stahlke2015}}]\label{prop:channelfromsourcegraph}
Let $(O_A,\Gamma)$ be a confusability graph. Then there exists a source $C: S \to O_A \otimes O_B$ such that $(O_A,\Gamma)$ is the confusability graph of the source $C$.
\end{proposition}
\begin{proof}
Let $O_A$ split as $O_A \cong Z_A \otimes Z_A^*$, where $Z_A: r_0 \to a$. We want to define Charlie's system $S\cong Z_S \otimes Z_S^*$, Bob's system $O_B \cong Z_B \otimes Z_B^*$, and a channel $C: S \to O_A \otimes O_B$ with dilation $\beta: Z_S \to (Z_A \boxtimes Z_B) \otimes E_C$ such that the positive 2-morphism $\widetilde{f} \in \End(Z_A^* \otimes Z_A)$ shown in~\eqref{eq:ftilde} has support $s(\widetilde{f}) = \widetilde{\Gamma^{\perp}}$.

\begin{itemize}
\item We define $S := \mathbb{C} \oplus \mathbb{C}$; that is, the direct sum of the trivial $G$-$C^*$-algebra with itself. This has dilation $Z_S: = \begin{pmatrix}\mathbbm{1} \\ \mathbbm{1}\end{pmatrix} : r_0 \to r_0 \boxplus r_0$, , where $\mathbbm{1}$ is the monoidal unit object of the $C^*$-tensor category $\End(r_0)$. (See~\cite[\S{}6.1]{Verdon2021} for a summary of matrix notation for 1-morphisms between direct sums of objects in $\TwoRep(G)$.)

\item We define Bob's system by $Z_B:= (Z_A \otimes Z_A^*)$. Since $Z_B$ has type $r_0 \to r_0$, the system $O_A \otimes O_B$ can be split as $Z_B \otimes Z_A$, since $Z_B \otimes Z_A \cong Z_A \boxtimes Z_B$. 

\item We define $E_C:= \begin{pmatrix}Z_A^* \\ Z_A^*\end{pmatrix}:a \to r_0 \boxplus r_0$. (Here we have used the isomorphism $Z_A \boxtimes Z_B \cong Z_B \otimes Z_A$.)

\item The dilation $\beta: Z_S \to Z_B \otimes Z_A \otimes E_C$ will have two components $\begin{pmatrix}\beta_0 \\ \beta_1\end{pmatrix}$, corresponding to the two factors of $r_0 \boxplus r_0$; these have type $\beta_0,\beta_1: \mathbbm{1} \to Z_A \otimes Z_A^* \otimes Z_A \otimes Z_A^*$. We define
\begin{align*}
\includegraphics[scale=.8]{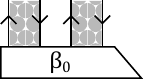}
~~:=~~
\includegraphics[scale=.8]{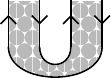}
&&
\includegraphics[scale=.8]{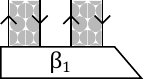}
~~:=~~
\includegraphics[scale=.8]{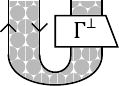}
\end{align*}
Here we have shaded the $a$-region using packed circles. It is clear by positivity and faithfulness of the trace that one can normalise $\beta_0$ and $\beta_1$ by a scalar to make them isometries and thus obtain a channel $C$. 
\end{itemize}
By Example~\ref{ex:classrel} we see that the complete simple graph on $\mathbbm{1} \oplus \mathbbm{1}$ is the complete simple graph on two points $0,1$ in the usual sense, i.e. with edges $0 \rightarrow 1$ and $1 \rightarrow 0$. The positive 2-morphism $\widetilde{f}$ is therefore as follows (up to irrelevant normalising scalars):
\begin{align*}
\includegraphics[scale=.8]{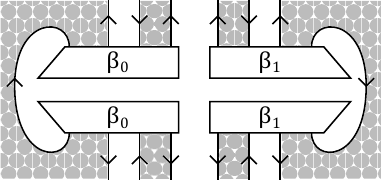}
~~+~~
\includegraphics[scale=.8]{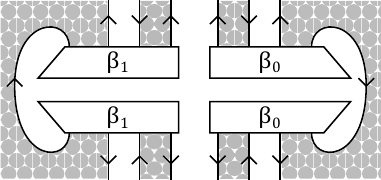}
\end{align*}
\begin{align*}
=~~
\includegraphics[scale=.8]{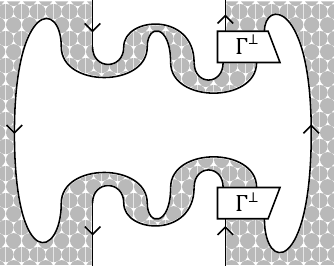}
~~+~~
\includegraphics[scale=.8]{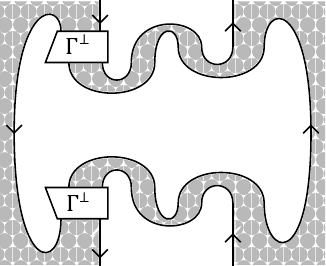}
\end{align*}
\begin{align*}
=~~
\includegraphics[scale=.8]{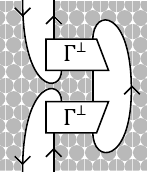}
~~+~~
\includegraphics[scale=.8]{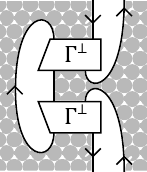}
~~=~~
2
\includegraphics[scale=.8]{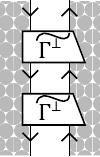}
~~=~~
2
\includegraphics[scale=.8]{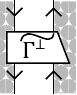}
\end{align*}
The result follows, since the support of the final expression is clearly $\widetilde{\Gamma^{\perp}}$. 
\end{proof}

\bibliographystyle{alphaurl}
\bibliography{bibliography}

\appendix
\section{Appendix}

\subsection{A definition of the 2-category $\TwoRep(G)$}\label{app:tworepdef}

In this appendix we will provide a self-contained definition of the 2-category $\TwoRep(G)$. For more detail see~\cite{Verdon2021}. 

\paragraph{Standard duals in $\Rep(G)$.} The category $\Rep(G)$ of finite-dimensional unitary representations of a compact quantum group $G$ is a rigid $C^*$-tensor category. In particular, every object $X$ has a \emph{standard} right dual; this is an object $X^*$, together with cup and cap morphisms $\eta: \mathbbm{1} \to X^* \otimes X$ and $\epsilon: X \otimes X^* \to \mathbbm{1}$  obeying the snake equations~\eqref{eq:snake} (here $\mathbbm{1}$ is the tensor unit of $\Rep(G)$). This right dual is also a left dual by the cup and cap $\epsilon^{\dagger}: \mathbbm{1} \to X \otimes X^*$ and $\eta^{\dagger}: X^* \otimes X \to \mathbbm{1}$. 

The standard dual $[X^*,\epsilon,\eta]$ is defined up to unitary equivalence by the property that for any morphism $f \in \End(X)$ the left trace is equal to the right trace:
\begin{equation}\label{eq:standardrepg}
\epsilon \circ (f \otimes \id_{X^*}) \circ \epsilon^{\dagger}
~=~
\eta^{\dagger} \circ (\id_{X^*} \otimes f) \circ \eta
\end{equation}
When we say `defined up to unitary equivalence', we mean that for any two duals $[X^*,\eta,\epsilon], [(X^*)',\eta',\epsilon']$ satisfying the condition~\eqref{eq:standardrepg} there exists a unitary $u: X^* \to (X^*)'$ such that 
\begin{align*}
\eta' = (u \otimes \id_X) \circ \eta 
&& 
\epsilon' = \epsilon \circ (\id_{X} \otimes u^{\dagger})
\end{align*}

\paragraph{Separable standard Frobenius algebras.} 
A \emph{Frobenius algebra} in $\Rep(G)$ is an object $A$ equipped with a \emph{multiplication} morphism $m: A \otimes A \to A$ and a \emph{unit} morphism $u: \mathbbm{1} \to A$, obeying the following conditions (which we call \emph{associativity}, \emph{unitality} and the \emph{Frobenius equation} respectively):
\begin{align*}
\includegraphics[scale=.8,valign=c]{pictures/tworepapp/frobassoc1.pdf}
~~=~~
\includegraphics[scale=.8,valign=c]{pictures/tworepapp/frobassoc2.pdf}
&&
\includegraphics[scale=.8,valign=c]{pictures/tworepapp/frobunit1.pdf}
~~=~~
\includegraphics[scale=.8,valign=c]{pictures/tworepapp/frobunit2.pdf}
~~=~~
\includegraphics[scale=.8,valign=c]{pictures/tworepapp/frobunit3.pdf}
&&
\includegraphics[scale=.8,valign=c]{pictures/tworepapp/frobfrob1.pdf}
~~=~~
\includegraphics[scale=.8,valign=c]{pictures/tworepapp/frobfrob2.pdf}
~~=~~
\includegraphics[scale=.8,valign=c]{pictures/tworepapp/frobfrob3.pdf}
\end{align*}
Here we have drawn the multiplication and unit $m,u$ and their daggers $m^{\dagger}: A \to A \otimes A$ and $u^{\dagger}: A \to \mathbb{C}$ (called the \emph{comultiplication} and the \emph{counit}) as white vertices; they can be distinguished by their type. 

We say that a Frobenius algebra is \emph{separable} if it obeys the following equation:
\begin{align*}
\includegraphics[scale=.8,valign=c]{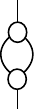}
~~=~~
\includegraphics[scale=.8,valign=c]{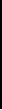}
\end{align*}
Every Frobenius algebra $A$ is self-dual, since the following cup and cap obey the snake equations~\eqref{eq:snake}:
\begin{align*}
\includegraphics[scale=.8,valign=c]{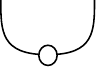}
~~:=~~
\includegraphics[scale=.8,valign=c]{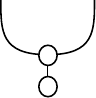}
&&
\includegraphics[scale=.8,valign=c]{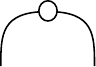}
~~:=~~
\includegraphics[scale=.8,valign=c]{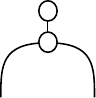}
\end{align*}
We say that the Frobenius algebra is \emph{standard} if this cup and cap form a standard duality --- that is, for any $f \in \End(A)$ the following equation is obeyed:
\begin{align*}
\includegraphics[scale=.8,valign=c]{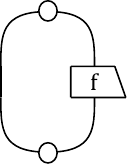}
~~=~~
\includegraphics[scale=.8,valign=c]{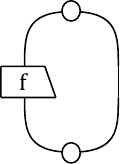}
\end{align*}
Separable standard Frobenius algebras ($\F$s) in $\Rep(G)$ correspond precisely to f.d.\ $G$-$C^*$-algebras. In one direction, a $G$-equivariant faithful positive linear functional on a f.d.\ $G$-$C^*$-algebra $A$ allows an inner product to be defined on that $C^*$-algebra, turning it into a f.d.\ unitary $G$-representation --- that is, an object $A$ of $\Rep(G)$~\cite[\S{}2.1]{Neshveyev2018}. The $G$-$C^*$-algebra structure yields multiplication  and unit morphisms $m: A \otimes A \to A$ and $u: \mathbbm{1} \to A$ in $\Rep(G)$ which satisfy the axioms of a Frobenius algebra. As was shown in~\cite[Thm. 2.11]{Neshveyev2018}, there is a unique choice of \emph{separable standard} $G$-equivariant faithful positive linear functional $\phi_A$ on any f.d.\ $G$-$C^*$-algebra $A$ such that the resulting Frobenius algebra is separable and standard (this functional is a renormalisation of the state in the cited theorem). The linear functional is recovered as the counit $u^{\dagger}: A \to \mathbb{C}$ of the resulting $\F$. For the other direction, see~\cite[\S{}3]{Verdon2020b}.

\paragraph{Dagger bimodules and bimodule morphisms.} Let $A,B$ be $\F$s in $\Rep(G)$. We say that an \emph{$A,B$-dagger bimodule} is an object $M$ of $\Rep(G)$ equipped with left and right \emph{action} maps $m_A: A \otimes M \to M$ and $m_B: M \otimes B \to M$ satisfying the following axioms:
\begin{align*}
\includegraphics[scale=.8,valign=c]{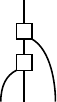}
~~=~~
\includegraphics[scale=.8,valign=c]{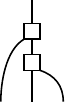}
&&
\includegraphics[scale=.8,valign=c]{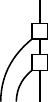}
~~=~~
\includegraphics[scale=.8,valign=c]{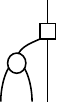}
&&
\includegraphics[scale=.8,valign=c]{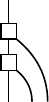}
~~=~~
\includegraphics[scale=.8,valign=c]{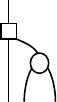}
\\
\includegraphics[scale=.8,valign=c]{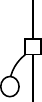}
~~=~~
\includegraphics[scale=.8,valign=c]{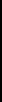}
~~=~~
\includegraphics[scale=.8,valign=c]{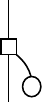}
&&
\includegraphics[scale=.8,valign=c]{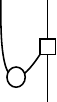}
~~=~~
\includegraphics[scale=.8,valign=c]{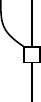}
&&
\includegraphics[scale=.8,valign=c]{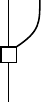}
~~=~~
\includegraphics[scale=.8,valign=c]{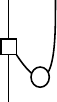}
\end{align*}
Here we have drawn the left and right actions $m_A$, $m_B$ and their daggers $m_A^{\dagger}: M \to A \otimes M$ and $m_B^{\dagger}: M \to M \otimes B$ as square vertices with $A$ or $B$ entering at a corner and $M$ passing through the vertical centre; they can then be distinguished by their type.

It is worth remarking here that $A,B$-dagger bimodules correspond to f.d.\ equivariant Hilbert $C^*$-bimodules between the f.d.\ $G$-$C^*$-algebras $A$ and $B$, as was observed in~\cite[\S{}2.3]{Neshveyev2018}. By `f.d.\ equivariant Hilbert $C^*$-bimodule' we mean a f.d.\ equivariant right Hilbert $C^*$-module $M$ over the $G$-$C^*$-algebra $B$ as defined in~\cite[\S{}1.2]{Neshveyev2018}, together with a covariant unital $*$-homomorphism from the $G$-$C^*$-algebra $A$ into the unital $G$-$C^*$-algebra of endomorphisms of the module $M$. Indeed, for any f.d.\ equivariant Hilbert $C^*$-bimodule $M$ the inner product $\phi_B(\braket{x|y})$, where $\phi_B$ is the separable standard functional on the $C^*$-algebra $B$, turns the $C^*$-bimodule into a unitary $G$-representation, i.e. an object of $\Rep(G)$. The equivariant Hilbert $C^*$-bimodule structure further induces an $A,B$-dagger bimodule structure on this object, where $A,B$ are now the $\F$s corresponding to the $G$-$C^*$-algebras. In the other direction, the $B$-valued inner product on an $A,B$-dagger bimodule is recovered as:
\begin{align*}
\braket{x|y} :=  \includegraphics[scale=.8,valign=c]{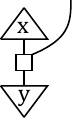}
\end{align*}
Let $M, N$ be $A,B$-dagger bimodules and let $f: M \to N$ be a morphism in $\Rep(G)$. We say that it is a \emph{bimodule morphism} if it intertwines the $A$- and $B$-actions:
\begin{calign}\label{eq:bimodule}
\includegraphics[scale=.8,valign=c]{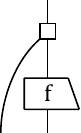}
~~=~~
\includegraphics[scale=.8,valign=c]{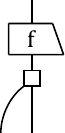}
&&
\includegraphics[scale=.8,valign=c]{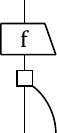}
~~=~~
\includegraphics[scale=.8,valign=c]{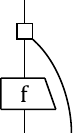}
\end{calign}

\paragraph{The interior tensor product of dagger bimodules.} Let $A,B,C$ be $\F$s, let $M$ be an $A,B$-dagger bimodule and let $N$ be a $B,C$-dagger bimodule. Their \emph{interior tensor product} $M \otimes_B N$ is a $A,C$-dagger bimodule defined as follows. The following morphism is an idempotent in $\End(M \otimes N)$:
\begin{align*}
\includegraphics[scale=.8,valign=c]{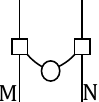}
\end{align*}
We split this idempotent to obtain an object $M \otimes_B N$ and an isometry $\iota: M \otimes_B N \to M \otimes N$; we will depict $\iota$ and its dagger $\iota^{\dagger}$ by an downwards and upwards-pointing triangle respectively. We have the following equations:
\begin{align*}
\includegraphics[scale=.8,valign=c]{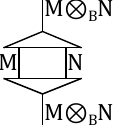}
~~=~~
\includegraphics[scale=.8,valign=c]{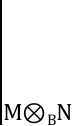}
&&
\includegraphics[scale=.8,valign=c]{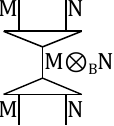}
~~=~~
\includegraphics[scale=.8,valign=c]{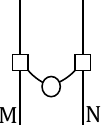}
\end{align*}
The left $A$-action and right $B$-action on $M \otimes_B N$ are defined as follows:
\begin{align*}
\includegraphics[scale=.8,valign=c]{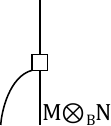}
~~:=~~
\includegraphics[scale=.8,valign=c]{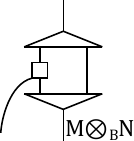}
&&
\includegraphics[scale=.8,valign=c]{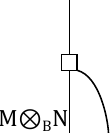}
~~:=~~
\includegraphics[scale=.8,valign=c]{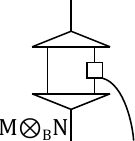}
\end{align*}
The interior tensor product extends to bimodule morphisms: given $A,B$-dagger bimodules $M,N$ and $B,C$-dagger bimodules $O,P$, and bimodule morphisms $f: M \to N$ and $g: O \to P$, there is a bimodule morphism $f \otimes g: M \otimes_B O \to N \otimes_B P$ defined as follows:
\begin{align*}
\includegraphics[scale=.8,valign=c]{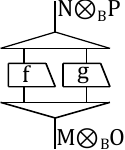}
\end{align*}

\paragraph{The 2-category $\TwoRep(G)$.} The $C^*$-2-category $\TwoRep(G)$ is defined as follows:
\begin{itemize}
\item Objects are $\F$s $A,B,\dots$ in $\Rep(G)$.
\item 1-morphisms $A \to B$ are $A,B$-dagger bimodules $M,N,\dots$. Composition of 1-morphisms is given by interior tensor product. 
\item 2-morphisms $M \to N$ are bimodule morphisms. Horizontal composition of 2-morphisms is given by interior tensor product. 
\end{itemize}
Let $\mathbbm{1}$ be the trivial $\F$ in $\Rep(G)$. (That is, the tensor unit equipped with the identity morphism as multiplication and unit; we assume WLOG that $\Rep(G)$ is strict.) There is an equivalence (in fact, an isomorphism) between the tensor category $\End(\mathbbm{1})$ in $\TwoRep(G)$ and the tensor category $\Rep(G)$: indeed, a $\mathbbm{1},\mathbbm{1}$-dagger bimodule is simply an object of $\Rep(G)$, and the intertwiner condition~\eqref{eq:bimodule} for a bimodule morphism is trivial. We can therefore consider $\TwoRep(G)$ as a sort of extension or completion of $\Rep(G)$, which `lives' inside $\TwoRep(G)$ as an endomorphism category. Indeed, $\TwoRep(G)$ is known as the \emph{$Q$-system completion} of $\Rep(G)$~\cite{Chen2022}. 

\subsection{Monoidal structure and traces on $\TwoRep(G)$}\label{app:monoidal}

We say that a compact quantum group is \emph{quasitriangular} if it has a unitary $R$-matrix~\cite[Def. 2.6.2]{Neshveyev2013}. In this case, the category $\Rep(G)$ of f.d.\ continuous unitary representations of $G$ is braided~\cite[P.82]{Neshveyev2013}. Without loss of generality we can take $\Rep(G)$ to be strict monoidal. We use the definition of $\TwoRep(G)$ as the 2-category of separable standard Frobenius algebras ($\F$s), dagger bimodules and bimodule homomorphisms in $\Rep(G)$; we refer to Appendix~\ref{app:tworepdef} or~\cite[\S{}3.1]{Verdon2021} for definitions and notation. 

We refer to~\cite[Def. 4.1]{Kapranov1994} for the definition of a monoidal 2-category; the $C^*$-version is identical except that the map $\boxtimes$ should be unitary and linear on 2-morphisms and the structure 2-isomorphisms should be unitary. We define the tensor product $\boxtimes$ on $\TwoRep(G)$ as follows:
\begin{itemize}
\item Tensor product of objects: $A \boxtimes B := A \otimes B$, i.e. the tensor product of $\F$s, whose multiplication and unit are as follows:
\begin{calign}\label{eq:prodalg}
\includegraphics[scale=.8,valign=c]{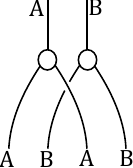}
&&
\includegraphics[scale=.8,valign=c]{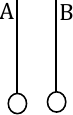}
\end{calign}
\item Tensor product of 1-morphisms: ${}_{A_1}M_{A_2} \boxtimes {}_{B_1}N_{B_2} := {}_{(A_1 \otimes B_1)}(M \otimes N)_{(A_2 \otimes B_2)}$, where the tensor product bimodule action is defined as follows:
\begin{calign}\label{eq:tpbimodact}
\includegraphics[scale=.8,valign=c]{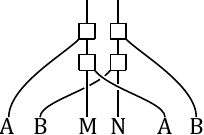}
\end{calign}
\item Tensor product of 2-morphisms: $f \boxtimes g := f \otimes g$. (This is clearly an intertwiner for the action~\eqref{eq:tpbimodact}.)
\item Monoidal unit for tensor product of objects: the trivial $\F$ on the simple unit object $\mathbbm{1}$.
\item First unitary tensorator ($(\rightarrow \otimes \rightarrow)$ in~\cite[Def. 4.1]{Kapranov1994}):
\begin{align*}
&\includegraphics[scale=.8,valign=c]{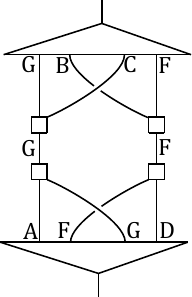}
\\
({}_A A_A \boxtimes {}_B F_{D}) \otimes_{A \otimes D} ({}_A G_{C} &\boxtimes {}_D D_D) \to 
({}_A G_{C} \boxtimes {}_B B_B ) \otimes_{C \otimes B} ({}_C C_C \boxtimes {}_B F_{D} )
\end{align*}
(Here the triangles are the projectors and inclusions for the tensor product of dagger bimodules in $\Rep(G)$; see~\cite[Eq. 28]{Verdon2021}. Recall that the identity 1-morphism $A \to A$ in $\TwoRep(G)$ is the bimodule ${}_A A_A$.)
\item Second unitary tensorator ($(\rightarrow \rightarrow \otimes ~\cdot~)$) in~\cite[Def. 4.1]{Kapranov1994}):
\begin{align*}
&\includegraphics[scale=.8,valign=c]{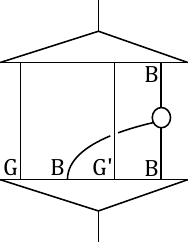}
\\
({}_A G_{A'} \boxtimes {}_B B_B) \otimes_{A' \otimes B} ({}_{A'} G'_{A''} &\boxtimes {}_B B_B) \to 
({}_A G_{A'} \otimes_{A'} {}_{A'} G'_{A''}) \boxtimes {}_B B_B
\end{align*}
\item Third unitary tensorator ($(~\cdot~ \otimes \rightarrow \rightarrow)$ in~\cite[Def. 4.1]{Kapranov1994}):
\begin{align*}
&\includegraphics[scale=.8,valign=c]{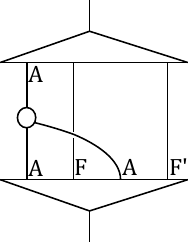}
\\
({}_A A_A \boxtimes {}_B F_{B'}) \otimes_{A \otimes B'} ({}_A A_A &\boxtimes {}_{B'} F'_{B''}) \to  {}_A A_A \boxtimes 
({}_A F_{A'} \otimes_{A'} {}_{A'} F'_{A''})
\end{align*}
\item All the other structure 2-morphisms are trivial, since we have taken $\Rep(G)$ to be strict. 
\end{itemize}
We leave the checks that the polyhedra in~\cite[Def. 4.1]{Kapranov1994} commute to the reader. We now define the trace and the partial trace.
\begin{definition}\label{def:endotrace}
Let $r$ be an object of $\TwoRep(G)$, corresponding to a $\F$ $A$. The trace $\Tr_r: \End(r) \to \mathbb{C}$ is defined as follows:
\begin{calign}\label{eq:endotracebimod}
f \in \End({}_{A} A_A) && \mapsto &&
\includegraphics[scale=.8,valign=c]{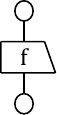}
\end{calign}
\end{definition}
\begin{remark}
Note that this is not the trace defined in~\cite[\S{}2.4]{Verdon2021}. In fact, this trace is a better choice even in the non-monoidal setting of that work; for instance, with this choice the cup and cap in~\cite[Eq. 41]{Verdon2021} are already spherical, so there is no need to normalise as in~\cite[Rem. 3.13, Eq. 51]{Verdon2021}. 
\end{remark}
\noindent
We recall a couple of definitions from~\cite[\S{}3.1.2]{Verdon2021}. Let $r,s$ be objects in $\TwoRep(G)$, corresponding to algebras $A,B$. For a 1-morphism $X: r \to s$ --- that is, an $(A,B)$-dagger bimodule ${}_A X_B$ --- the dual ${}_B(X^*)_{A}: s \to r$ is the $(B,A)$-dagger bimodule on the object $X^*$ in $\Rep(G)$, with the following actions of $B$ and $A$:
\begin{align*} 
\includegraphics[scale=.7,valign=c]{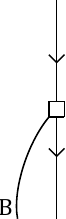}
~~:=~~
\includegraphics[scale=.7,valign=c]{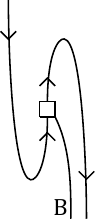}
&&
\includegraphics[scale=.7,valign=c]{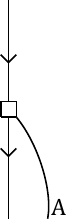}
~~:=~~
\includegraphics[scale=.7,valign=c]{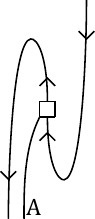}
\end{align*}
Here the left and right cups and caps we have used to define the action on the dual are the cups and caps of the standard duality in $\Rep(G)$. 

Having defined our dual 1-morphism, we now define a right cup 2-morphism $\eta_{{}_A X_B}:  {}_B B_B \to  {}_B(X^*) \otimes_A X_B $ and a right cap 2-morphism $\epsilon_{{}_A X_B}: {}_A X \otimes_B (X^*)_A \to {}_A A_A$ witnessing the duality:
\begin{calign}\label{eq:cupscapsbimod}
\includegraphics[scale=.7,valign=c]{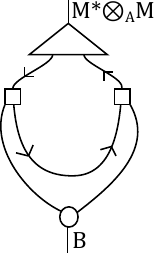} &&
\includegraphics[scale=.7,valign=c]{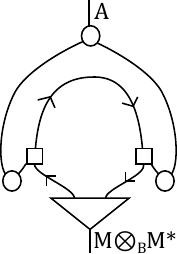}
\end{calign}
The left cup and cap are the daggers of these 2-morphisms. 

\begin{lemma}\label{lem:endotrace}
Let $r,s$ be any objects of $\TwoRep(G)$. For any 1-morphism $X: r \to s$ and any $f \in \End(X)$ we have the following equality:
\begin{calign}\label{eq:tracestoprove}
\Tr_s\left(
\includegraphics[valign=c]{pictures/2cats/ltrace.pdf}\right)
~=~
\Tr_r\left(\includegraphics[valign=c]{pictures/2cats/rtrace.pdf}\right)
\end{calign}
(Here the diagrams are drawn in $\TwoRep(G)$.) The resulting map $\Tr: \End(X) \to \mathbb{C}$ is a positive faithful trace on $\End(X)$.
\end{lemma}
\begin{proof}
Let $A,B$ be the algebras corresponding to $r$ and $s$ respectively. Unpacking the expressions in~\eqref{eq:tracestoprove} using~\eqref{eq:endotracebimod} and~\eqref{eq:cupscapsbimod}, the equality~\eqref{eq:tracestoprove} is seen as follows:
\begin{calign}\label{eq:spherical}
\includegraphics[scale=.7,valign=c]{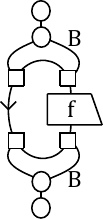}
~~=~~
\includegraphics[scale=.7,valign=c]{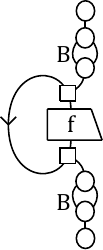}
~~=~~
\includegraphics[scale=.7,valign=c]{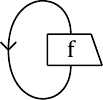}
~~=~~
\includegraphics[scale=.7,valign=c]{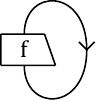}
~~=~~
\includegraphics[scale=.7,valign=c]{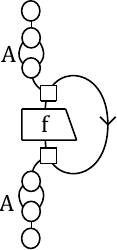}
~~=~~
\includegraphics[scale=.7,valign=c]{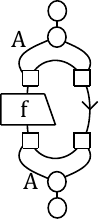}
\end{calign}
Here for the middle equality we used sphericality of the standard duality in $\Rep(G)$. Since the middle two expressions are simply the standard trace of $f$ as a morphism in $\Rep(G)$, we see that $\Tr$ is also positive and faithful.
\end{proof}
\noindent
Finally, we define the partial trace. 
\begin{definition}\label{def:partialtrace}
For an object $A \boxtimes B$ in $\TwoRep(G)$ and a morphism $X \in \End(A)$ we define the \emph{partial trace} $\Tr_s: \End(X \boxtimes {}_B B_B) \to \End(X)$ as follows:
\begin{align*}
\Tr_s(f)~~:=~~
\includegraphics[scale=.8,valign=c]{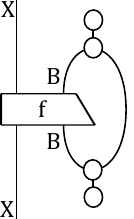}
\end{align*}
\end{definition}
\begin{lemma}\label{lem:partialtrace}
The partial trace is faithful and positive, and $\Tr_r \circ \Tr_s = \Tr_{r \boxtimes s}$.
\end{lemma}
\begin{proof}
The partial trace is clearly positive. Since we saw in~\eqref{eq:spherical} that the trace $\Tr(f)$ is just the trace as a morphism in $\Rep(G)$, it follows that $\Tr_r \circ \Tr_s = \Tr_{r \boxtimes s}$, since the $\F$ is standard and the tensor product dual~\cite[Eq. 4]{Verdon2021} is standard. It then follows that the partial trace is faithful, because we know that $\Tr_{r \boxtimes s}$ is faithful, and if $\Tr_s(f) = 0$ then $\Tr_{r \boxtimes s} = 0$.
\end{proof}

\end{document}